\documentclass{amsart}
\usepackage{amssymb}
\usepackage{mathrsfs}
\usepackage{cases}
\usepackage{amsmath}
\usepackage{amsfonts}
\usepackage{pifont}
\usepackage{graphicx,amssymb,mathrsfs,amsmath}
\usepackage{tocvsec2}
\newtheorem{thm}{Theorem}[section]

\newtheorem{prop}[thm]{Proposition}

\theoremstyle{definition}
\newtheorem{defn}[thm]{Definition}
\newtheorem{example}[thm]{Example}
\theoremstyle{remark}
\newtheorem{rem}[thm]{Remark}
\numberwithin{equation}{section}

\begin{document}
\title[Metrical almost periodicity: ...]{Metrical almost periodicity: Levitan and Bebutov concepts}

\author{Belkacem Chaouchi}
\address{Lab. de l'Energie et des Syst\`{e}mes Intelligents,
Khemis Miliana University, 44225, Algeria}
\email{chaouchicukm@gmail.com}

\author{Marko Kosti\' c}
\address{Faculty of Technical Sciences,
University of Novi Sad,
Trg D. Obradovi\' ca 6, 21125 Novi Sad, Serbia}
\email{marco.s@verat.net}

\author{Daniel Velinov}
\address{Department for Mathematics, Faculty of Civil Engineering, Ss. Cyril and Methodius University, Skopje,
Partizanski Odredi
24, P.O. box 560, 1000 Skopje, N. Macedonia}
\email{velinovd@gf.ukim.edu.mk}

{\renewcommand{\thefootnote}{} \footnote{
\\ \text{  }  \ \    {\it Key words and phrases.} 
Levitan metrical almost periodicity, Bebutov metrical almost periodicity, abstract Volterra integro-differential equations.
\\  \text{  }  
This research is partially supported by grant 174024 of
Ministry of Science and Technological Development, Republic of Serbia and
Bilateral project between MANU and SANU.}}

\begin{abstract}
In this paper, we analyze Levitan and Bebutov metrical approximations of functions $F :\Lambda \times X \rightarrow Y$ by trigonometric polynomials and $\rho$-periodic type functions, where $\emptyset \neq \Lambda \subseteq {\mathbb R}^{n}$, $X$ and $Y$ are complex Banach spaces, and $\rho$ is a general binary relation on $Y$. We also analyze various classes of multi-dimensional Levitan almost periodic functions in general metric and multi-dimensional Bebutov uniformly recurrent functions in general metric. 
We provide several applications of our theoretical results to
the abstract Volterra integro-differential equations and the partial differential equations.
\end{abstract}
\maketitle

\section{Introduction and preliminaries}

The notion of almost periodicity was introduced by the Danish mathematician H. Bohr around 1924-1926 and later generalized by many other authors (see the research monographs \cite{besik}, \cite{diagana}, \cite{fink}, \cite{gaston}, \cite{nova-mono}, \cite{nova-selected}, \cite{188}, \cite{levitan}, \cite{pankov} and \cite{30} for more details about almost periodic functions and their applications).
Suppose that $(X,\| \cdot \|)$ is a complex Banach space, and $F : {\mathbb R}^{n} \rightarrow X$ is a continuous function, where $n\in {\mathbb N}$. Then it is said that $F(\cdot)$ is almost periodic if and only if for each $\epsilon>0$
there exists a finite real number $l>0$ such that for each ${\bf t}_{0} \in {\mathbb R}^{n}$ there exists ${\bf \tau} \in B({\bf t}_{0},l)\equiv \{ {\bf t} \in {\mathbb R}^{n} : |{\bf t}-{\bf t}_{0}|\leq l\}$ such that
\begin{align*}
\bigl\|F({\bf t}+{\bf \tau})-F({\bf t})\bigr\| \leq \epsilon,\quad {\bf t}\in {\mathbb R}^{n};
\end{align*}
here, $|\cdot -\cdot|$ denotes the Euclidean distance in ${\mathbb R}^{n}.$
Equivalently, $F(\cdot)$ is almost periodic if and only if for any sequence $({\bf b}_k)$ in ${\mathbb R}^{n}$ there exists a subsequence $({\bf a}_{k})$ of $({\bf b}_k)$
such that the sequence of translations $(F(\cdot+{\bf a}_{k}))$ converges in $C_{b}({\mathbb R}^{n}: X),$ the Banach space of all bounded continuous functions on ${\mathbb R}^{n},$ equipped with the sup-norm. Any trigonometric polynomial in ${\mathbb R}^{n}$ is almost periodic, and a continuous function $F(\cdot)$ is almost periodic if and only if there exists a sequence of trigonometric polynomials in ${\mathbb R}^{n}$ which converges uniformly to $F(\cdot)$. 

A continuous function $F : {\mathbb R}^{n} \rightarrow X$ is said to be uniformly recurrent (uniformly Poisson stable) if and only if there exists a sequence $(\tau_{k})$ in ${\mathbb R}^{n}$
such that $\lim_{k\rightarrow +\infty}|\tau_{k}|=+\infty$ and $\lim_{k\rightarrow +\infty}F({\bf t}+\tau_{k})=F({\bf t}),$ uniformly in ${\bf t}\in {\mathbb R}^{n}$ (uniformly in ${\bf t}\in K$, for any compact subset $K\subseteq {\mathbb R}^{n}$). Any almost periodic function is uniformly recurrent and any uniformly recurrent function is uniformly Poisson stable, while the converse statements are not true in general; cf. also \cite{marko-manuel-ap,rho,metrical} for some recent results about almost periodic functions and uniformly recurrent functions.

Further on, let $F : {\mathbb R}^n \rightarrow X$ be a continuous function. Then it is said that the function $F(\cdot)$ is almost automorphic if and only if for every sequence $({\bf b}_{k})$ in $\mathbb{R}^n$ there exist a subsequence $({\bf a}_{k})$ of $({\bf b}_{k})$ and a mapping $G : {\mathbb R}^n \rightarrow X$ such that
\begin{align}\label{first-equ}
\lim_{k\rightarrow \infty}F\bigl( {\bf t}+{\bf a}_{k}\bigr)=G({\bf t})\ \mbox{ and } \  \lim_{k\rightarrow \infty}G\bigl( {\bf t}-a_{k}\bigr)=F({\bf t}),
\end{align}
pointwisely for ${\bf t}\in {\mathbb R}^n.$ The range of an almost automorphic function $F(\cdot)$ is relatively compact in $X,$ and the limit function $G(\cdot)$ is bounded on ${\mathbb R}^n$ but not necessarily continuous on ${\mathbb R}^n.$ If the convergence of limits appearing in \eqref{first-equ} is uniform on compact subsets of ${\mathbb R}^n,$ then we say that the function $F(\cdot)$ is compactly almost automorphic. It is well known that an almost automorphic function $F(\cdot)$ is compactly almost automorphic if and only if $F(\cdot)$ is uniformly continuous (cf. \cite{marko-manuel-aa} for further information about the multi-dimensional almost automorphic functions and their applications). 

On the other hand, the class of almost periodic functions can be generalized following the approach of B. M. Levitan (see, e.g., \cite{levitan} for the one-dimensional setting): Suppose that $F : {\mathbb R}^{n} \rightarrow X$ is a continuous function, $N>0$ and $\epsilon>0.$ Then a point  $\tau \in {\mathbb R}^{n}$ is said to be an $\epsilon,N$-almost period of function $F(\cdot)$ 
if and only if 
\begin{align*}
\bigl\|F({\bf t}+{\bf \tau})-F({\bf t})\bigr\| \leq \epsilon \mbox{ for all } {\bf t}\in {\mathbb R}^{n}\mbox{ with }|{\bf t}|\leq N;
\end{align*}
denote by $E(\epsilon, N)$ the set consisting of all $\epsilon,N$-almost periods of function $F(\cdot)$. 
Let us say, maybe for the first time in the existing literature, that a continuous function $F(\cdot)$ is Levitan pre-almost periodic if and only if for each  $N>0$ and $\epsilon>0,$ 
there exists a finite real number $l>0$ such that for each ${\bf t}_{0} \in {\mathbb R}^{n}$ there exists ${\bf \tau} \in B({\bf t}_{0},l) \cap E(\epsilon,N),$ i.e., the set $E(\epsilon,N)$ is relatively dense in ${\mathbb R}^{n}$ for each $N>0$ and $\epsilon>0.$ It is worth noting that B. Ya. Levin has shown, in \cite{levin}, that the sum of two Levitan
pre-almost periodic functions $f: {\mathbb R}\rightarrow {\mathbb R}$ and
$g: {\mathbb R}\rightarrow {\mathbb R}$
need not be Levitan pre-almost periodic, in general. In the definition of a Levitan ($N$-)almost periodic function $F :  {\mathbb R}^{n} \rightarrow X,$ we additionally require that, for every real numbers $N>0$ and $\epsilon>0,$ there exist a finite real number $\eta >0$ and a relatively dense set $E_{\eta;N}$ of $(\eta,N)$-almost periods of $F(\cdot)$ such that $E_{\eta;N} \pm E_{\eta;N} \subseteq E(\epsilon, N);$ cf. \cite[condition (2), p. 54, l. 3]{levitan} and the corresponding footnote for more details concerning this issue in the one-dimensional setting. Due to R. Yuan's result \cite[Theorem 3.1]{r-yuan-jde}, we know that a bounded continuous function $f: {\mathbb R}\rightarrow X$ is compactly almost automorphic if and only if $f(\cdot)$ is Levitan $N$-almost periodic (cf. also 
B. Basit's paper \cite{basit-trst}, A. Reich \cite{reich} and references cited therein for more details about the relationship between the almost automorphic functions and the Levitan $N$-almost periodic functions on topological groups).

The notion of a recurrent
function in the continuous Bebutov system \cite{bebutov} is based on the use of 
topology of uniform convergence on compact sets
(cf. also Subsection 2.3.9 in the monograph \cite{bertoti} by G. Bertotti and I. D. Mayergoyz,
the paper \cite{328} by L. I. Danilov and references cited therein for further information
in this direction). A uniformly recurrent function is also called pseudo-periodic by H.
Bohr, which has been accepted by many other authors later
on; a recurrent function in the continuous Bebutov system is also called (uniformly)
Poisson-stable motion by M. V. Bebutov. The Levitan almost periodic solutions and the uniformly Poisson stable solutions for various classes of (abstract) differential equations have been sought in many research articles by now; see, e.g., \cite{carpati,cheban,chebban,x-liu,lubar,nawr,scherb4,scherb5,scherb6} and references cited therein; it is also worth noting that M. Akhmet, M.
Tleubergenova and A. Zhamanshin have recently analyzed the existence and uniqueness of 
modulo periodic Poisson stable
solutions of quasi-linear differential
equations in \cite{entropy}.

Concerning the Poisson stability of motions of dynamical systems and solutions of
differential equations, we would like to specifically mention the research monographs \cite{scherb1}-\cite{scherb2} by B. A. Shcherbakov.
It is also worth noticing that T. Caraballo and D. Cheban have analyzed the existence and uniqueness of 
Levitan/Bohr almost periodic (almost automorphic)
solutions of the second-order monotone differential equations in \cite{cheban-lev}. The Poisson stability of motions for monotone nonautonomous dynamical systems and
of solutions for certain classes of monotone nonautonomous differential equations has been analyzed by D. Cheban and Z. Liu in \cite{cheban-lev1}; cf. also \cite{cheban-sch,scherb1,scherb2,scherb3,scherb6,sell-mono,shen-yi} and references cited therein for related results obtained within the theory of dynamical systems. Finally, let us note that the interpolation by Levitan almost periodic functions was considered by S. Hartman in \cite{hartmanlev} (1974), while the difference property for perturbations of vector-valued Levitan
almost periodic functions was considered by B. Basit and H. G\"unzler in \cite{basit-2005}; see also 
M. G. Lyubarskii \cite{lubar}.

On the other hand, the first systematic study of metrical almost periodicity was conducted by the second named author in 2021 (\cite{metrical}). The Stepanov, Weyl and Besicovitch classes of metrical $\rho$-almost periodic type functions have been considered in \cite{metrical-stepanov}, \cite{metrical-weyl} and \cite{multi-besik}, respectively. In our joint research article \cite{chaouchi-metr} with B. Chaouchi and D. Velinov, we have recently studied the metrical approximations of functions  $F :\Lambda \times X \rightarrow Y$ by trigonometric polynomials and $\rho$-periodic type functions, where $\emptyset \neq \Lambda \subseteq {\mathbb R}^{n}$, $X$ and $Y$ are complex Banach spaces, and $\rho$ is a general binary relation on $Y;$ we have also introduced and analyzed various classes of almost periodic functions and uniformly recurrent functions in general metric therein. 

The main aim of this research article is to continue the analysis raised in \cite{chaouchi-metr} by investigating the Levitan and Bebutov metrical approximations of functions $F :\Lambda \times X \rightarrow Y$ by trigonometric polynomials and $\rho$-periodic type functions. As mentioned in the abstract, we analyze here various classes of multi-dimensional Levitan almost periodic functions in general metric and multi-dimensional Bebutov uniformly recurrent functions in general metric. 

The organization and main ideas of this paper can be briefly described as follows. After explaining the notation and terminology used throughout the paper as well as the main concepts necessary for understanding anything that follows, we introduce the basic function spaces of metrically almost periodic functions in the sense of Levitan/Bebutov approach in the second section of paper; cf. Definition \ref{strong-appl-lev} and Definition \ref{dfggg-metkl}-Definition \ref{nafaks-metkl}. The main structural results established in this section are Proposition \ref{nafaks-metkljkl}, Proposition \ref{asad-lev} and Proposition \ref{rep}; cf. also Example \ref{kuchi}, Example \ref{kuchi123} and Example \ref{ferplay}. Levitan $(N,c)$-almost periodic functions and uniformly Poisson $c$-stable functions [multi-dimensional Levitan $N$-almost periodic functions] are specifically analyzed in Subsection \ref{zajeg} [Subsection \ref{kura-tura}]. In Subsection \ref{kura-tura}, we introduce the spaces of strongly Levitan $N$-almost periodic functions of type $1,$
the Levitan $N$-almost periodic functions of type $1$ and
the strongly Levitan $N$-almost periodic functions. We show that these spaces have the linear vector structure and propose an open problem whether the Levitan $N$-almost periodic functions form a vector space with the usual operations if $n\geq 2.$ By a simple counterexample, we show that the Bogolybov theorem (see, e.g., \cite[pp. 55-57]{levitan}) cannot be straightforwardly extended to the higher-dimensional case $n\geq 2$; cf. also Proposition \ref{bengeder} and Example \ref{lalala}.

In Section \ref{bebprim}, we present several applications of our theoretical results to the abstract Volterra integro-differential equations. Subsection \ref{tackice} is devoted to the study of invariance of Levitan $N$-almost like periodicity under the actions of the infinite convolution products and certain applications to the abstract Cauchy problems without initial conditions; Subsection \ref{216} investigates the convolution invariance of certain kinds of multi-dimensional Levitan $N$-almost periodic type functions (the second application of this subsection, concerning the fractional diffusion-wave equations with Caputo-Dzhrbashyan fractional derivatives, and the third application of this subsection, concerning the biharmonic partial differential operator, are essentially new and not considered anywhere else in the existing literature). In this subsection, we also reconsider several results established recently by A. Nawrocki in \cite{nawr}.  Subsection \ref{21666} continues our analysis of
the wave equation in ${\mathbb R}^{n}$; the final section of paper is reserved for some comments and final remarks about the introduced spaces of Levitan $N$-almost periodic type functions. We propose several open problems to our readers, providing also a great number of important references concerning the subjects under our considerations.
\vspace{1.6pt}

\noindent {\bf Notation and terminology.} Suppose that $X,\ Y,\ Z$ and $ T$ are given non-empty sets. Let us recall that a binary relation between $X$ and $Y$
is any subset
$\rho \subseteq X \times Y.$ 
If $\rho \subseteq X\times Y$ and $\sigma \subseteq Z\times T$ with $Y \cap Z \neq \emptyset,$ then
we define
$\sigma \cdot  \rho =\sigma \circ \rho \subseteq X\times T$ by
$
\sigma \circ \rho :=\{(x,t) \in X\times T : \exists y\in Y \cap Z\mbox{ such that }(x,y)\in \rho\mbox{ and }
(y,t)\in \sigma \}.
$
As is well known, the domain and range of $\rho$ are defined by $D(\rho):=\{x\in X :
\exists y\in Y\mbox{ such that }(x,y)\in X\times Y \}$ and $R(\rho):=\{y\in Y :
\exists x\in X\mbox{ such that }(x,y)\in X\times Y\},$ respectively; $\rho (x):=\{y\in Y : (x,y)\in \rho\}$ ($x\in X$), $ x\ \rho \ y \Leftrightarrow (x,y)\in \rho .$
If $\rho$ is a binary relation on $X$ and $n\in {\mathbb N},$ then we define $\rho^{n}
$ inductively. Set $\rho (X'):=\{y : y\in \rho(x)\mbox{ for some }x\in X'\}$ ($X'\subseteq X$).

We will always assume henceforth that $(X,\| \cdot \|)$ and $(Y, \|\cdot\|_Y)$ are complex Banach spaces, $n\in {\mathbb N},$ $\emptyset  \neq \Lambda \subseteq {\mathbb R}^{n},$ and
${\mathcal B}$ is a non-empty collection of non-empty subsets of $X$ satisfying
that
for each $x\in X$ there exists $B\in {\mathcal B}$ such that $x\in B.$ For the sequel, we set
$$
\Lambda'':=\bigl\{ \tau \in {\mathbb R}^{n} : \tau +\Lambda \subseteq \Lambda \bigr\}.
$$
By
$L(X,Y)$ we denote the Banach space of all bounded linear operators from $X$ into
$Y,$ $L(X,X)\equiv L(X);$ ${\mathrm I}$ denotes the identity operator on $Y.$ 
Define 
${\mathbb N}_{n}:=\{1,..., n\}.$ If ${\mathrm A}$ and ${\mathrm B}$ are non-empty sets, then we define ${\mathrm B}^{{\mathrm A}}:=\{ f | f : {\mathrm A} \rightarrow {\mathrm B}\}.$ By $\| \cdot \|_{\infty}$ we denote the sup-norm; the symbol $f_{|K}(\cdot)$ denotes the restriction of a function $f(\cdot)$ to a non-empty subset $K$ of its domain.

Suppose now that $\nu : \Lambda\rightarrow (0,\infty)$ and the function $1/\nu(\cdot)$ is locally bounded. Then the vector space $C_{b,\nu}(\Lambda : Y)$ consists of all continuous functions $u : \Lambda \rightarrow
Y$ such that $\sup_{{\bf t}\in \Lambda}\|u({\bf t})\|_{Y}\nu({\bf t})<+\infty$. Equipped with the norm
$\|\cdot\|:=\sup _{{\bf t}\in \Lambda}\|\nu({\bf t}) \cdot({\bf t})\|_{Y},$ $C_{b,\nu}(\Lambda : Y)$ is a Banach space.

We need to recall the following notion (\cite{nova-selected}):

\begin{defn}\label{drasko-presing}
\begin{itemize}
\item[(i)]
Let ${\bf \omega}\in {\mathbb R}^{n} \setminus \{0\},$ $\rho$ be a binary relation on $Y$ 
and 
${\bf \omega}\in \Lambda''$. A continuous
function $F:\Lambda \times X\rightarrow Y$ is said to be $({\bf \omega},\rho)$-periodic [$\rho$-periodic] if and only if 
$
F({\bf t}+{\bf \omega};x)\in \rho(F({\bf t};x)),$ ${\bf t}\in \Lambda,$ $x\in X$ [there exists ${\bf \omega}\in ({\mathbb R}^{n} \setminus \{0\}) \cap \Lambda''$ such that $F(\cdot;\cdot)$ is $({\bf \omega},\rho)$-periodic]. 
\item[(ii)]
Let ${\bf \omega}_{j}\in {\mathbb R} \setminus \{0\},$ $\rho_{j}\in {\mathbb C} \setminus \{0\}$ be a binary relation on $Y$
and 
${\bf \omega}_{j}e_{j}+\Lambda \subseteq \Lambda$ ($1\leq j\leq n$). A continuous
function $F:\Lambda \times X \rightarrow Y$ is said to be $({\bf \omega}_{j},\rho_{j})_{j\in {\mathbb N}_{n}}$-periodic [$(\rho_{j})_{j\in {\mathbb N}_{n}}$-periodic] if and only if 
$
F({\bf t}+{\bf \omega}_{j}e_{j};x)\in \rho_{j}(F({\bf t};x)),$ ${\bf t}\in \Lambda,
$ $x\in X,$ $j\in {\mathbb N}_{n}$ [there exist non-zero real numbers $\omega_{j}$ such that ${\bf \omega}_{j}e_{j}\in \Lambda''$ for all $j\in {\mathbb N}_{n}$ and $F(\cdot;\cdot)$ is $({\bf \omega}_{j},\rho_{j})_{j\in {\mathbb N}_{n}}$-periodic].
\item[(iii)] Let ${\bf \omega}_{j}\in {\mathbb R} \setminus \{0\},$ 
and 
${\bf \omega}_{j}e_{j}+\Lambda \subseteq \Lambda$ ($1\leq j\leq n$). A continuous
function $F:\Lambda \times X \rightarrow Y$ is said to be periodic if and only if $F(\cdot;\cdot)$ is $(\rho_{j})_{j\in {\mathbb N}_{n}}$-periodic with $\rho_{j}={\rm I}$ for all $j\in {\mathbb N}_{n}.$
\end{itemize}
\end{defn} \index{function!$({\bf \omega}_{j},\rho_{j})_{j\in {\mathbb N}_{n}}$-periodic}

Finally, let us recall that a trigonometric polynomial $P : \Lambda \times X \rightarrow Y$ is any linear combination of functions like
\begin{align*}
e^{i[\lambda_{1}t_{1}+\lambda_{2}t_{2}+\cdot \cdot \cdot +\lambda_{n}t_{n}]}c(x),
\end{align*}
where $\lambda_{i}$ are real numbers ($1\leq i \leq n$) and $c: X \rightarrow Y$ is a continuous mapping.

\section{Metrical approximations: Levitan and Bebutov concepts}\label{docolord}

In this section, we assume that $\emptyset \neq \Lambda \subseteq {\mathbb R}^{n}$ and
$\phi : [0,\infty) \rightarrow [0,\infty).$ 
If $\emptyset \neq K\subseteq {\mathbb R}^{n}$ is a compact set and $K\cap \Lambda \neq \emptyset$, then we assume that the function ${\mathbb F}_{K} : K\cap \Lambda \rightarrow [0,\infty)$ and the pseudometric space
${\mathcal P}_{K}=(P_{K},d_{K})$, where $P_{K}\subseteq [0,\infty)^{K\cap \Lambda}$, are given.
Define $\| g\|_{P_{K}}:=d_{K}(g,0)$ for any $g\in P_{K}.$

The following notion plays an important role in our analysis (cf. also \cite[Definition 2.1]{chaouchi-metr}):

\begin{defn}\label{strong-appl-lev} 
Suppose that $\emptyset  \neq \Lambda \subseteq {\mathbb R}^{n}$ and $F :\Lambda  \times X \rightarrow Y.$
Then we say that $F(\cdot;\cdot)$ is:
Levitan strongly $(\phi,{\mathbb F}_{\bf K},{\mathcal B},{\mathcal P}_{\bf K})$-almost periodic (Levitan semi-\\$(\phi,\rho,{\mathbb F}_{\bf K},{\mathcal B},{\mathcal P}_{\bf K})$-periodic, Levitan semi-$(\phi,\rho_{j},{\mathbb F}_{\bf K},{\mathcal B},{\mathcal P}_{\bf K})_{j\in {\mathbb N}_{n}}$-periodic) if and only if for each $B\in {\mathcal B}$ and for each non-empty compact set $ K\subseteq {\mathbb R}^{n}$ such that $K\cap \Lambda \neq \emptyset$ there exists a sequence ($P_{k}^{B,K}:  \Lambda \times X \rightarrow Y$) of trigonometric polynomials ($\rho$-periodic functions ($P_{k}^{B,K}:  \Lambda \times X \rightarrow Y$), $(\rho_{j})_{j\in {\mathbb N}_{n}}$-periodic functions $(P_{k}^{B,K}:  \Lambda \times X \rightarrow Y$))
such that ${\mathbb F}_{K}(\cdot)[ \phi(\|P_{k}^{B,K}(\cdot;x)-F(\cdot;x)\|_{Y})]_{| K\cap \Lambda} \in P_{K}$ for all $x\in X,$ and
\begin{align*}
\lim_{k\rightarrow +\infty}\sup_{x\in B} \Biggl\| {\mathbb F}_{K}(\cdot) \Bigl[ \phi\Bigl(\bigl\|P_{k}^{B,K}(\cdot;x)-F(\cdot;x)\bigr\|_{Y}\Bigr)\Bigr]_{| K\cap \Lambda}\Biggr\|_{P_{K}}=0.
\end{align*}
\end{defn}\index{function!Levitan strongly $(\phi,{\mathbb F}_{\bf K},{\mathcal B},{\mathcal P}_{\bf K})$-almost periodic} \index{function!Levitan semi-$(\phi,\rho,{\mathbb F}_{\bf K},{\mathcal B},{\mathcal P}_{\bf K})$-periodic}\index{function!Levitan semi-$(\phi,\rho_{j},{\mathbb F}_{\bf K},{\mathcal B},{\mathcal P}_{\bf K})_{j\in {\mathbb N}_{n}}$-periodic}

As in our former research studies, we omit the term
``$\phi$''
if $\phi(x)\equiv x,$
``$\rho$'' if $\rho={\rm I},$ the term ``${\mathcal B}$'' if $X=\{0\}$ and ${\mathcal B}=\{X\}.$ We will also omit the term ``${\mathbb F}_{\bf K}$'' if ${\mathbb F}_{\bf K}\equiv 1$ for each non-empty compact set $ K\subseteq {\mathbb R}^{n}$ such that $K\cap \Lambda \neq \emptyset $ (in the case that 
$\phi(x)\equiv x,$ $\rho={\rm I},$ $P_{K}=C_{b}(K\cap \Lambda)$ and ${\mathbb F}_{\bf K}\equiv 1$ for each non-empty compact set $ K\subseteq {\mathbb R}^{n}$ such that $K\cap \Lambda \neq \emptyset ,$ the notion of Levitan strong $(\phi,{\mathbb F}_{\bf K},{\mathcal B},{\mathcal P}_{\bf K})$-almost periodicity is completely regardless; for example, any continuous function $F :  {\mathbb R}^{n} \rightarrow Y$ is
Levitan strongly $(x,{\mathbb F}_{\bf K},{\mathcal P}_{\bf K})$-almost periodic due to the Weierstrass approximation theorem).

The following result can be proved in a similar fashion as \cite[Proposition 2.3]{chaouchi-metr}:

\begin{prop}\label{kimihap}
Suppose that $\emptyset  \neq \Lambda \subseteq {\mathbb R}^{n}$, $F :\Lambda  \times X \rightarrow Y$, $h : Y \rightarrow Z$ is Lipschitz continuous, $\phi(\cdot)$ is monotonically increasing and there exists a function $\varphi : [0,\infty) \rightarrow [0,\infty)$ such that $\phi(xy)\leq \varphi(x)\phi(y)$ for all $x,\ y\geq 0.$ Let the assumptions \emph{(C0)}-\emph{(C1)} hold for every non-empty compact set $ K\subseteq {\mathbb R}^{n}$ with $K\cap \Lambda \neq \emptyset ,$ where:\index{condition!(C1)}\index{condition!(C2)}
\begin{itemize}
\item[(C0-K)]
The assumptions $0\leq f\leq g$ and $g\in P_{K}$ imply $f\in P_{K}$ and $\|f\|_{P_{K}}\leq \|g\|_{P_{K}}.$
\item[(C1-K)]
If
$f\in P_{K},$ then $d'f\in P_{K}$ for all reals $d'\geq 0,$ and there exists a finite real constant $d_{K}>0$ such that $\|d'f\|_{P_{K}}\leq d_{K}(1+d')\|f\|_{P_{K}}$ for all reals $d'\geq 0$ and all functions $f\in P_{K}.$
\end{itemize}
Then we have the following:
\begin{itemize}
\item[(i)]
Suppose that $F(\cdot;\cdot)$ is Levitan semi-$(\phi,\rho,{\mathbb F}_{\bf K},{\mathcal B},{\mathcal P}_{\bf K})$-periodic (Levitan semi-$(\phi,\rho_{j},{\mathbb F}_{\bf K},{\mathcal B},{\mathcal P}_{\bf K})_{j\in {\mathbb N}_{n}}$-periodic), and
$h\circ \rho \subseteq \rho \circ h$ ($h\circ \rho_{j} \subseteq \rho_{j} \circ h$ for $1\leq j\leq n$).
Then the function $h\circ F : \Lambda \times X \rightarrow Y$ is likewise Levitan semi-$(\phi,\rho,{\mathbb F}_{\bf K},{\mathcal B},{\mathcal P}_{\bf K})$-periodic (Levitan semi-$(\phi,\rho_{j},{\mathbb F}_{\bf K},{\mathcal B},{\mathcal P}_{\bf K})_{j\in {\mathbb N}_{n}}$-periodic). 
\item[(ii)] Suppose that $X=\{0\}$,  there exists a finite real constant $c>0$ such that $\phi(x+y)\leq c[\varphi(x)+\varphi(y)]$ for all $x,\ y\geq 0,$ $\phi(\cdot)$ is continuous at the point zero and, for every non-empty compact set $ K\subseteq {\mathbb R}^{n}$ with $K\cap \Lambda \neq \emptyset ,$ we have
\begin{align*}
{\mathbb F}_{K}\in P_{K}\ \ \mbox{ and }\lim_{\epsilon \rightarrow 0+}\|\epsilon {\mathbb F}_{K}(\cdot)\|_{P_{K}}=0.
\end{align*}
Suppose, further, that the function $F(\cdot)$ is Levitan strongly $(\phi,{\mathbb F}_{\bf K},{\mathcal B},{\mathcal P}_{\bf K})$-almost periodic and
the assumption \emph{(C2-K)} holds, where:
\begin{itemize}
\item[(C2-K)]
There exists a finite real constant $e_{K}>0$ such that the assumptions $f,\ g\in P_{K}$ and $0\leq w\leq d'[f+g]$ for some finite real constant $d'>0$ imply $w\in P_{K}$ and $\| w\|_{P_{K}}\leq e_{K}(1+d')[\|f\|_{P_{K}}+\|g\|_{P_{K}}].$ 
\end{itemize}
Then the function $(h\circ F)(\cdot)$ is likewise Levitan strongly $(\phi,{\mathbb F}_{\bf K},{\mathcal B},{\mathcal P}_{\bf K})$-almost periodic.
\end{itemize}
\end{prop}

We continue by providing the following illustrative example:

\begin{example}\label{dzo-mudonj}
Let us recall that A. Haraux and P. Souplet have proved, in \cite[Theorem 1.1]{haraux}, that the function
\begin{align}\label{jednac}
f(t):=\sum_{m=1}^{\infty}\frac{1}{m}\sin^{2}\Bigl( \frac{t}{2^{m}} \Bigr),\quad t\in {\mathbb R},
\end{align}\index{function!slowly $p$-semi-periodic in variation}
is not Besicovitch-$p$-almost periodic for any finite exponent $p\geq 1$, as well as that $f(\cdot)$ is uniformly recurrent and uniformly continuous (cf. also \cite{nova-mono}).
We also know that the function $f(\cdot)$ is slowly $p$-semi-periodic in variation ($1\leq p<+\infty$); see \cite{chaouchi-metr} for the notion and more details. Let $\epsilon_{0}>0$ be a fixed real number; using the elementary inequality $|\sin t| \leq |t|,$ $t\in {\mathbb R},$ it readily follows that the function $f(\cdot)$ is Levitan  semi-$(x,{\rm I},{\mathbb F}_{\bf K},{\mathcal P}_{K})$-semi-periodic with ${\mathbb F}_{[-N,N]}\equiv N^{-2-\epsilon_{0}}$ and ${\mathcal P}_{[-N,N]}:=C[-N,N]$ for $N>0.$ 

Similarly, any continuous function $F: {\mathbb R}^{n} \rightarrow {\mathbb C}$ 
given by $F({\bf t})=\sum_{m=1}^{\infty}P_{m}({\bf t}),$ ${\bf t}\in {\mathbb R}^{n}$, where $P_{m}(\cdot)$ is a
trigonometric polynomial ($m\in {\mathbb N}$), is Levitan  semi-$(x,{\rm I},{\mathbb F}_{\bf K},{\mathcal P}_{K})$-semi-periodic with the constant function
${\mathbb F}_{[-N,N]^{n}}\equiv c_{N}>0$ appropriately chosen and  ${\mathcal P}_{[-N,N]^{n}}:=C([-N,N]^{n})$ for $N>0,$ provided that there exist a summable sequence of non-negative real numbers $(a_{m})$ and a continuous function $G :  {\mathbb R}^{n} \rightarrow {\mathbb C}$ such that $|F({\bf t})| \leq |G({\bf t})| \cdot \sum_{m=1}^{\infty}a_{m},$ ${\bf t}\in {\mathbb R}^{n}.$
\end{example}

We continue by introducing the following notion (cf. also \cite[Definition 2.5, Definition 2.8]{chaouchi-metr}):

\begin{defn}\label{dfggg-metkl}
Suppose that ${\mathrm R}$ is any collection of sequences in $\Lambda'',$ $F: \Lambda \times X \rightarrow Y$ and $\phi : [0,\infty) \rightarrow [0,\infty).$ Then we say that the function $F(\cdot;\cdot)$ is Levitan
$(\phi,{\mathrm R}, {\mathcal B},{\mathbb F}_{\bf K},{\mathcal P}_{\bf K})$-normal
if and only if for every set $B\in {\mathcal B},$ for every non-empty compact set $ K\subseteq {\mathbb R}^{n}$ with $K\cap \Lambda \neq \emptyset$ and for every sequence $({\bf b}_{k})_{k\in {\mathbb N}}$ in ${\mathrm R},$ there exists a subsequence $({\bf b}_{k_{m}})_{m\in {\mathbb N}}$ of $({\bf b}_{k})_{k\in {\mathbb N}}$ such that, for every $\epsilon>0,$ there exists an integer $m_{0}\in {\mathbb N}$ such that, for every integers $m,\ m'\geq m_{0},$ we have ${\mathbb F}_{K}(\cdot)[\phi( \| F(\cdot+{\bf b}_{k_{m}};x)-F(\cdot+{\bf b}_{k_{m'}};x)\|_{Y} )]_{| K\cap \Lambda}\in P_{K}$ for all $x\in X,$ and
\begin{align*}
\sup_{x\in B}\Biggl\|{\mathbb F}_{K}(\cdot)\Bigl[ \phi\Bigl( \bigl\| F(\cdot+{\bf b}_{k_{m}};x)-F(\cdot+{\bf b}_{k_{m'}};x)\bigr\|_{Y} \Bigr) \Bigr]_{| K\cap \Lambda}\Biggr\|_{P_{K}} <\epsilon.
\end{align*}
\end{defn}\index{function!$(\phi,{\mathrm R}, {\mathcal B},{\mathbb F}_{\bf K},{\mathcal P}_{\bf K})$-normal}

\begin{defn}\label{nafaks-metkl}
Suppose that $\emptyset  \neq \Lambda' \subseteq {\mathbb R}^{n},$ $\emptyset  \neq \Lambda \subseteq {\mathbb R}^{n},$ $F : \Lambda \times X \rightarrow Y$ is a given function, $\rho$ is a binary relation on $Y,$ and $\Lambda' \subseteq \Lambda''.$ Then we say that:
\begin{itemize}
\item[(i)]\index{function!Levitan $(\phi,{\mathbb F}_{\bf K},{\mathcal B},\Lambda',\rho,{\mathcal P}_{\bf K})$-almost periodic}
$F(\cdot;\cdot)$ is Levitan $(\phi,{\mathbb F}_{\bf K},{\mathcal B},\Lambda',\rho,{\mathcal P}_{\bf K})$-almost periodic if and only if for every $B\in {\mathcal B}$, $\epsilon>0$ and for every non-empty compact set $ K\subseteq {\mathbb R}^{n}$ with $K\cap \Lambda \neq \emptyset$,
there exists $l>0$ such that for each ${\bf t}_{0} \in \Lambda'$ there exists ${\bf \tau} \in B({\bf t}_{0},l) \cap \Lambda'$ such that, for every ${\bf t}\in \Lambda$ and $x\in B,$ there exists an element $y_{{\bf t};x}\in \rho (F({\bf t};x))$ such that ${\mathbb F}_{K}(\cdot) [\phi(\| F(\cdot+{\bf \tau};x)-y_{\cdot;x}\|_{Y})]_{| K\cap \Lambda}\in P_{K}$ for all $x\in X,$ and
\begin{align*}
\sup_{x\in B} \Biggl\|{\mathbb F}_{K}(\cdot) \Bigl[ \phi\Bigl(\bigl\| F(\cdot+{\bf \tau};x)-y_{\cdot;x}\bigr\|_{Y}\Bigr)\Bigr]_{| K\cap \Lambda}\Biggr\|_{P_{K}} \leq \epsilon .
\end{align*}
\item[(ii)] \index{function!Bebutov $(\phi,{\mathbb F}_{\bf K},{\mathcal B},\Lambda',\rho,{\mathcal P}_{\bf K})$-uniformly recurrent}\index{function!Bebutov $(\phi,{\mathbb F}_{\bf K},{\mathcal B},\Lambda',\rho,{\mathcal P}_{\bf K})$-uniformly recurrent of type $1$}
$F(\cdot;\cdot)$ is Bebutov $(\phi,{\mathbb F}_{\bf K},{\mathcal B},\Lambda',\rho,{\mathcal P}_{\bf K})$-uniformly recurrent if and only if for every $B\in {\mathcal B}$ and for every non-empty compact set $ K\subseteq {\mathbb R}^{n}$ with $K\cap \Lambda \neq \emptyset$, 
there exists a sequence $({\bf \tau}_{k})$ in $\Lambda'$ such that $\lim_{k\rightarrow +\infty} |{\bf \tau}_{k}|=+\infty$ and that, for every ${\bf t}\in \Lambda$ and $x\in B,$ there exists an element $y_{{\bf t};x}\in \rho (F({\bf t};x))$ such that ${\mathbb F}_{K}(\cdot)[\phi(\|F(\cdot+{\bf \tau}_{k};x)-y_{\cdot;x}\|_{Y})]_{| K\cap \Lambda}\in P_{K}$ for all $x\in X$, and
\begin{align*}
\lim_{k\rightarrow +\infty}\sup_{x\in B} \Biggl\|{\mathbb F}_{K}(\cdot) \Bigl[\phi\Bigl(\bigl\|F(\cdot+{\bf \tau}_{k};x)-y_{\cdot;x}\bigr\|_{Y}\Bigr)\Bigr]_{| K\cap \Lambda}\Biggr\|_{P_{K}}=0.
\end{align*}
If the sequence $({\bf \tau}_{k})$ in $\Lambda'$ is independent of the choice of a non-empty compact set $ K\subseteq {\mathbb R}^{n}$ with $K\cap \Lambda \neq \emptyset$, for a set $B\in {\mathcal B}$ given in advance,
then we say that the function $F(\cdot;\cdot)$ is Bebutov $(\phi,{\mathbb F}_{\bf K},{\mathcal B},\Lambda',\rho,{\mathcal P}_{\bf K})$-uniformly recurrent of type $1.$
\end{itemize}
\end{defn}

In any normal situation, a Bebutov $(\phi,{\mathbb F}_{\bf K},{\mathcal B},\Lambda',\rho,{\mathcal P}_{\bf K})$-uniformly recurrent function is already Bebutov $(\phi,{\mathbb F}_{\bf K},{\mathcal B},\Lambda',\rho,{\mathcal P}_{\bf K})$-uniformly recurrent function of type $1$: 

\begin{prop}\label{nafaks-metkljkl}
Suppose that $\emptyset  \neq \Lambda' \subseteq {\mathbb R}^{n}$ and $\emptyset  \neq \Lambda \subseteq {\mathbb R}^{n}$ are unbounded sets, $F : \Lambda \times X \rightarrow Y$ is a given function, $\rho$ is a binary relation on $Y,$ and $\Lambda' \subseteq \Lambda''.$ If $F(\cdot;\cdot)$ is Bebutov $(\phi,{\mathbb F}_{\bf K},{\mathcal B},\Lambda',\rho,{\mathcal P}_{\bf K})$-uniformly recurrent, then
$F(\cdot;\cdot)$ is Bebutov $(\phi,{\mathbb F}_{\bf K},{\mathcal B},\Lambda',\rho,{\mathcal P}_{\bf K})$-uniformly recurrent of type $1$, provided that for each compact set $K\subseteq {\mathbb R}^{n}$ such that $K\cap \Lambda \neq \emptyset,$ there exists a finite real constant $c_{K}>0$ such that, for every $x\in X,\ \tau \in \Lambda',$ and for every compact set $K'\subseteq {\mathbb R}^{n}$ which contains $K,$ we have
\begin{align}
\notag \Biggl\|& {\mathbb F}_{K}(\cdot) \Bigl[\phi\Bigl(\bigl\|F(\cdot+{\bf \tau};x)-y_{\cdot;x}\bigr\|_{Y}\Bigr)\Bigr]_{| K\cap \Lambda}\Biggr\|_{P_{K}}
\\\label{poginuo}& \leq c_{K}\Biggl\| {\mathbb F}_{K'}(\cdot) \Bigl[\phi\Bigl(\bigl\|F(\cdot+{\bf \tau}_;x)-y_{\cdot;x}\bigr\|_{Y}\Bigr)\Bigr]_{| K'\cap \Lambda}\Biggr\|_{P_{K'}},\quad y_{\cdot;x}\in \rho(F(\cdot;x)).
\end{align}
\end{prop}

\begin{proof}
Let the set $B\in {\mathcal B}$ be given. We know that there exists a natural number $N_{0}\in {\mathbb N}$ such that $[-N,N]^{n}\cap \Lambda \neq \emptyset $ for every natural number $N\geq N_{0}.$ Choose a point $\tau_{N}$ in $\Lambda'$ such that $|{\bf \tau}_{N}|\geq N$ and
\begin{align*}
\sup_{x\in B}\Biggl\| {\mathbb F}_{[-N,N]^{n}}(\cdot) \Bigl[\phi\Bigl(\bigl\|F(\cdot+{\bf \tau}_{N};x)-y_{\cdot;x}\bigr\|_{Y}\Bigr)\Bigr]_{| [-N,N]^{n}\cap \Lambda}\Biggr\|_{P_{[-N,N]^{n}}}\leq 1/N.
\end{align*}
Then the sequence $(\tau_{N})$ in $\Lambda'$ satisfies the desired requirements since we have assumed \eqref{poginuo}.
\end{proof}

The notion introduced in \cite[Definition 2.1, Definition 2.5, Definition 2.8]{chaouchi-metr} is a special case of the notion introduced in the previous three definitions, provided that
for each non-empty compact set $ K\subseteq {\mathbb R}^{n}$ with $K\cap \Lambda \neq \emptyset$ we have the existence of a finite real constant $c_{K}>0$ such that $d(f_{|K\cap \Lambda},g_{|K\cap \Lambda})\leq c_{K}d(f,g)$ for all $f,\ g\in P,$ 
${\mathbb F}_{K}={\mathbb F}_{| K\cap \Lambda},$
and $P_{K}=\{ f_{| K\cap \Lambda} ;  f\in P\};$
here and hereafter, ${\mathcal P}=(P,d) $ is a pseudometric space and $P\subseteq [0,\infty)^{Y}$ contains the zero function.

Usually, we plug ${\mathbb F}_{K}(\cdot)\equiv 1,$ $\phi(x) \equiv x,$ $\rho={\mathrm I}$ and $P_{K}=C(\Lambda \cap K:Y)$ in Definition \ref{dfggg-metkl} and Definition \ref{nafaks-metkl}. For some examples of the one-dimensional uniformly Poisson stable functions, we refer the reader to \cite{scherb3} and the research monograph \cite[pp. 219-223]{vries} by J. de Vries (the function constructed by B. A. Shcherbakov in \cite{scherb3} is uniformly recurrent, in fact, which simply follows from an application of \cite[Lemma, p. 324]{scherb3}). Any compactly almost automorphic function $f : {\mathbb R} \rightarrow X$ is bounded, uniformly continuous and Levitan $N$-almost periodic, which simply implies that $F(\cdot)$ is uniformly Poisson stable; on the other hand, we know that there exists a compactly almost automorphic function $f : {\mathbb R} \rightarrow {\mathbb R}$ which is not (asymptotically) uniformly recurrent (see, e.g., \cite[Example 2.4.35]{nova-selected}). Therefore, the class of compactly almost automorphic functions seems to be ideal for finding certain functions which are uniformly Poisson stable but not uniformly recurrent (the function analyzed in Example \ref{dzo-mudonj} is uniformly recurrent but not Stepanov almost automorphic; cf. \cite{nova-selected} for the notion).

\begin{example}\label{kuchi}
Let $c_{0}$ be the Banach space of all numerical sequences vanishing at infinity, equipped with the sup-norm. Define $f : [0,\infty) \rightarrow c_{0}$ by
$$
f(t):=\Biggl( \frac{4n^{2}t^{2}}{(t^{2}+n^{2})^{2}} \Biggr)_{n\in {\mathbb N}},\quad t\geq 0.
$$
Besides many other features, we know that the function $f(\cdot)$ is bounded, uniformly continuous, quasi-asymptotically almost periodic, and not almost automorphic; see, e.g., \cite[Example 3.11.14]{FKP}. Here we will prove that the function $f(\cdot)$ is not uniformly Poisson stable. Let us assume the contrary, and let 
\begin{align}\label{poludjeh}
0<\epsilon<\inf_{n\in {\mathbb N}}\frac{4n^{4}}{(2n^{2}+2n+1)^{2}};
\end{align}
then there exists a strictly increasing sequence $(\tau_{k})$ of positive real numbers such that $|f(\tau_{k})-f(0)|\leq \epsilon,$ $k\geq k_{0}$ for some positive integer $k_{0}\in {\mathbb N}.$ This implies
$$
\sup_{n\in {\mathbb N}}\frac{4n^{2}\tau_{k}^{2}}{(\tau_{k}^{2}+n^{2})^{2}}\leq \epsilon,\quad k\geq k_{0}.
$$
Let $\tau_{k}\in [n_{k},n_{k}+1)$ for some $n_{k}\in {\mathbb N}$ ($k\geq k_{0}$). Then the previous estimate implies
\begin{align*}
\epsilon \geq   \sup_{n\in {\mathbb N}}\frac{4n^{2}\tau_{k}^{2}}{(\tau_{k}^{2}+n^{2})^{2}}\geq 
\frac{4n_{k}^{2}n_{k}^{2}}{((n_{k}+1)^{2}+n_{k}^{2})^{2}}>\epsilon,
\end{align*}
which is a contradiction; see \eqref{poludjeh}.
\end{example}

In the following example, the considered metric space $P_{K}$ is different from $C(\Lambda \cap K:Y):$

\begin{example}\label{kuchi123}
Suppose that  $(T(t))\subseteq L(X,Y)$ is a strongly continuous operator family and ${\mathcal B}$ denotes the collection of all bounded subsets of $X.$ Define
$$
F(t,s;x):=e^{\int_{s}^{t}\varphi(\tau)\, d\tau}T(t-s)x,\quad (t,s)\in {\mathbb R}^{2}, \ x\in X.
$$
Then our analysis from \cite[Example 8.1.5]{nova-selected} shows the following: If the function $\varphi(\cdot)$ is bounded and Levitan $(x,{\mathrm R},1,{\mathcal P}_{K})$-normal, where ${\mathrm R}$ denotes the collection of all sequences in ${\mathbb R}$ and $P_{K}=L^{1}({\mathbb R})$ for each compact set $K\subseteq {\mathbb R},$ then the function $F(\cdot,\cdot;\cdot)$ is Levitan $(x,{\mathrm R}_{1},{\mathcal B},1,{\mathcal P}_{1,K})$-normal, where ${\mathrm R}_{1}$ denotes the collection of all sequences in $\{(x,x) : x\in {\mathbb R}\} $ and $P_{1,K}=C({\mathbb R})$ for each compact set $K\subseteq {\mathbb R}^{2}.$
\end{example}

The usually considered spaces of Levitan $N$-almost periodic functions and uniformly Poisson stable functions are translation invariant; the basic properties of multi-dimensional $\rho$-almost periodic functions clarified in \cite[Proposition 2.11]{rho} can be formulated in our new setting, as well. Further on,
the qualitative behaviour of the limit function of a uniformly convergent sequence of almost periodic type functions has been analyzed many times before. For example, using the same argumentation as in the proof of \cite[Proposition 2.6]{chaouchi-metr}, we can deduce the following result:

\begin{prop}\label{asad-lev}
Suppose that ${\mathrm R}$ is any collection of sequences in $\Lambda'',$ $F_{j}: \Lambda \times X \rightarrow Y,$ and the function $F_{j}(\cdot;\cdot)$ is
Levitan
$(\phi,{\mathrm R}, {\mathcal B},{\mathbb F}_{\bf K},{\mathcal P}_{\bf K})$-normal for all $j\in {\mathbb N}.$ If $F: \Lambda \times X \rightarrow Y$ and, for every set $B\in {\mathcal B}$, for every sequence  $({\bf b}_{k})_{k\in {\mathbb N}}$ in ${\mathrm R}$ and for every non-empty compact set $ K\subseteq {\mathbb R}^{n}$ with $K\cap \Lambda \neq \emptyset$, we have
\begin{align*}
\lim_{(j,k)\rightarrow +\infty}\sup_{x\in B}\Biggl\|{\mathbb F}_{K}(\cdot)\Bigl[\phi\Bigl(\bigl\| F_{j}(\cdot +{\bf b}_{j};x)-F(\cdot +{\bf b}_{k};x)\bigr\|_{Y}\Bigr)\Bigr]_{K\cap \Lambda} \Biggr\|_{P_{K}}=0,
\end{align*}
then the function $F(\cdot;\cdot)$ is likewise Levitan
$(\phi,{\mathrm R}, {\mathcal B},{\mathbb F}_{\bf K},{\mathcal P}_{\bf K})$-normal, provided that:
\begin{itemize}
\item[(i)] The function $\phi(\cdot)$ is monotonically increasing and there exists a finite real constant $c>0$ such that $\phi(x+y)\leq c[\phi(x)+\phi(y)]$ for all $x,\ y\geq 0.$
\item[(ii)] Condition \emph{(C3-K)} holds, where:
\begin{itemize}
\item[(C3-K)] For every non-empty compact set $ K\subseteq {\mathbb R}^{n}$ with $K\cap \Lambda \neq \emptyset$, 
there exists a finite real constant $f_{K}>0$ such that the assumptions $f,\ g,\ h\in P_{K}$ and $0\leq w\leq d'[f+g+h]$ for some finite real constant $d'>0$ imply $w\in P_{K}$ and $\| w\|_{P_{K}}\leq f_{K}(1+d')[\|f\|_{P_{K}}+\|g\|_{P_{K}}+\|h\|_{P_{K}}].$
\end{itemize}
\end{itemize}
\end{prop}\index{condition!(C3-K)}

Without going into full details, we will only note here that the supremum formula can be formulated in our new framework; for example, if a function $F : {\mathbb R}^{n} \rightarrow X$ is uniformly Poisson stable, then for each positive real number $a>0$ we have $\sup_{{\bf t}\in {\mathbb R}^{n}}\| F({\bf t})\| =\sup_{{\bf t}\in {\mathbb R}^{n};|{\bf t}|\geq a}\| F({\bf t})\|\in [0,\infty];$
cf. also \cite[Proposition 2.4.13, Proposition 6.1.6, Proposition 8.1.15]{nova-selected} for the case that $\rho={\rm I},$ and \cite[Proposition 2.13]{rho} for the case in which $\rho=T\in L(Y)$ is a linear isomorphism (the statement of \cite[Proposition 2.20]{rho} concerning multi-dimensional $\rho$-almost periodic functions in the finite-dimensional spaces can be formulated in our new setting, as well).

We continue with some examples:

\begin{example}\label{ferplay}
\begin{itemize}
\item[(i)] The spaces of Levitan $N$-almost periodic functions (Bebutov uniformly recurrent functions) $F: {\mathbb R}^{n} \rightarrow Y$ can be constricted if we use the pseudometric spaces ${\mathcal P}_{K}$ such that ${\mathcal P}_{K}$ is continuously embedded into the space $C(K : Y)$ for every non-empty compact set $K\subseteq {\mathbb R}^{n}$. For instance, the function $f(t):=1/(2+\cos t+ \cos (\sqrt{2}t)),$ $t\in {\mathbb R}$
is unbounded, continuous and Levitan $(x,1,{\mathcal P}_{\bf K})$-almost periodic, where ${\mathcal P}_{\bf K}=C(K)$; see \cite[pp. 58-59]{levitan} for more details. Therefore, the function $f(\cdot)$ is uniformly Poisson stable, as well; before going any further, we would like to ask whether the function $f(\cdot)$ is uniformly recurrent or Besicovitch $p$-bounded for some finite exponent $p\geq 1$ (cf. \cite{nova-selected} for the notion)?

Here,
we can also use the metric spaces ${\mathcal P}_{K}$ equipped with the distance of the form 
\begin{align*}
d\bigl(f_{| K},g_{| K}\bigr):=\sup_{{\bf t}\in K}\|f({\bf t})-g({\bf t})\|+d_{1}\bigl(f_{|K},g_{|K}\bigr),\quad f,\ g\in P,
\end{align*}
where $P$ is a subspace of the space $C_{b}({\mathbb R}^{n} : [0,\infty)),$ $d_{1}(\cdot;\cdot)$ is a pseudometric on $P_{K},$ ${\mathbb F}_{K}={\mathbb F}_{| K\cap \Lambda}$ and $P_{K}=\{ f_{| K\cap \Lambda} ;  f\in P\}.$ 
\item[(ii)] As already emphasized in \cite[Example 2.4(iii)]{chaouchi-metr}, the spaces of Levitan $N$-almost periodic functions (Bebutov uniformly recurrent functions) $F: {\mathbb R}^{n} \rightarrow Y$ can be extended if we use the pseudometric spaces ${\mathcal P}_{K}$ such that the space $C(K : Y)$ is continuously embedded into ${\mathcal P}_{K}$ for for every non-empty compact set $K\subseteq {\mathbb R}^{n}$. The use of the incomplete metric space $P$ consisting of all continuous functions from ${\mathbb R}^{n}$ into $[0,\infty),$ equipped with the distance
$$
d(f,g):=\sup_{x\in {\mathbb R}^{n}}\bigl| \arctan (f(x))-\arctan(g(x)) \bigr|,\quad f,\ g\in P,
$$
has been proposed in \cite{chaouchi-metr}. This concept allows one to consider the generalized almost periodicity of functions 
$f(\cdot)$ which are not locally integrable. We can similarly analyze the generalized Levitan $N$-almost periodicity and the generalized Bebutov uniform recurrence of the functions 
$f(\cdot)$ which are not locally integrable by replacing the term $\| \cdot \|_{P_{K}}$ by $\| \cdot \|_{\overline{P_{K}}},$ where $\overline{{\mathcal P}_{K}}=(\overline{P_{K}},\overline{d_{K}})$ is the completion of the metric space ${\mathcal P}_{K}.$ 
\item[(iii)] Define the function $f : {\mathbb R} \rightarrow {\mathbb R}$ by $f(x):=n3^{n+1}\sin(2\pi x)$ if $x\in [3^{n},3^{n}+1]+2\cdot 3^{n+1}{\mathbb Z}$ for some $n\in {\mathbb N}$, and $f(x):=0,$ otherwise. Then $f(\cdot)$ is clearly unbounded; moreover, we know that $f(\cdot)$ is Levitan $N$-almost periodic (see A. Nawrocki \cite[Lemma 2.18, Example 2.19]{nawr}). 
\item[(iv)] Suppose that $(\alpha_{k})$ is a fixed sequence of positive real numbers such that $\lim_{k\rightarrow +\infty}\alpha_{k}=+\infty.$ If $f_{i} : {\mathbb R}\rightarrow {\mathbb C}$ is a bounded, uniformly Poisson stable function such that $\lim_{k\rightarrow +\infty}f_{i}(\cdot+\alpha_{k})=f_{i}(\cdot),$ uniformly on compacts of ${\mathbb R}$ ($1\leq i\leq n$), then the function $F(\cdot),$ given by
$$
F\bigl( t_{1},...,t_{n} \bigr):=f_{1}\bigl( t_{1} \bigr)\cdot ... \cdot f_{n}\bigl( t_{n} \bigr),\quad {\bf t}=\bigl( t_{1},...,t_{n} \bigr)\in {\mathbb R}^{n},
$$ 
is bounded, uniformly Poisson stable and satisfies $\lim_{k\rightarrow +\infty}F(\cdot+\beta_{k})=F(\cdot),$ uniformly on compacts of ${\mathbb R}^{n},$ where $\beta_{k}=(\alpha_{k},...,\alpha_{k})$ for all $k\in {\mathbb N}.$
\item[(v)] The introduced notion has the meaning even if $\rho=0\in L(Y).$ In this case, we can simply construct a Bebutov $(x,1,0,{\mathcal P}_{\bf K})$-uniformly recurrent function $f : {\mathbb R} \rightarrow {\mathbb R}$ which does not vanish as $|t|\rightarrow +\infty$; here, $P_{K}=C(K)$ for each non-empty compact subset of ${\mathbb R}.$
\item[(vi)] In order to avoid any form of plagiarism, we will only note here that \cite[Example 6.1.13, Example 6.1.15, Example 6.1.16]{nova-selected} can be formulated in our new framework; these examples justify the introduction of our concepts with the set $\Lambda'$ being not equal to the set $\Lambda$ or some of its proper subsets.
\end{itemize}
\end{example}

In the usually considered situation, the notion introduced in Definition \ref{strong-appl-lev} is more specific than the notion introduced in Definition \ref{dfggg-metkl} and Definition \ref{nafaks-metkl}. Concerning this issue, we will clarify the following result (cf. also \cite[Proposition 2.9]{chaouchi-metr}):

\begin{prop}\label{rep}
\begin{itemize}
\item[(i)] Suppose that ${\mathrm R}$ is any collection of sequences in $\Lambda'',$ $F: \Lambda \times X \rightarrow Y,$ and $\phi : [0,\infty) \rightarrow [0,\infty).$ Let the following conditions hold:
\begin{itemize}
\item[(a)] The function $\phi(\cdot)$ is monotonically increasing, continuous at the point zero, and there exists a finite real constant $c>0$ such that $\phi(x+y)\leq c[\phi(x)+\phi(y)]$ for all $x,\ y\geq 0.$
\item[(b)] Condition \emph{(C3-K)} holds.
\item[(c)] For every non-empty compact set $ K\subseteq {\mathbb R}^{n}$ with $K\cap \Lambda \neq \emptyset$, we have that ${\mathbb F}_{K}(\cdot)\phi( \| P(\cdot;x) \|_{Y})\in P_{K}$ for any trigonometric polynomial (periodic function) $P(\cdot;\cdot)$ and $x\in X.$
\item[(d)]  For every non-empty compact set $ K\subseteq {\mathbb R}^{n}$ with $K\cap \Lambda \neq \emptyset$, there exists a finite real constant $g_{K}>0$ such that 
$$
\Bigl\|{\mathbb F}_{K}(\cdot)\Bigl[\phi \bigl( \| P(\cdot;x) \|_{Y}\bigr)\Bigr]_{| K\cap \Lambda }\Bigr\|_{P_{K}} \leq g_{K} \sup_{{\bf t}\in K\cap \Lambda} \phi \bigl( \| P({\bf t};x) \|_{Y}\bigr),
$$ 
for any any trigonometric polynomial (periodic function) $P(\cdot;\cdot)$ and $x\in X.$
\item[(e)]  For every non-empty compact set $ K\subseteq {\mathbb R}^{n}$ with $K\cap \Lambda \neq \emptyset$, there exists a finite real constant $h_{K}>0$ such that, for every $x\in X$ and $\tau \in \Lambda'',$ the assumptions $ H : \Lambda \times X \rightarrow Y$ and ${\mathbb F}_{K}(\cdot)[\phi(\| H(\cdot ; x)\|_{Y})]_{|K\cap \Lambda} \in P_{K}$ imply
${\mathbb F}_{K}(\cdot)[\phi(\| H(\cdot+\tau ;x)\|_{Y})]_{|K\cap \Lambda} \in P_{K}$ and 
$$
\Bigl\| {\mathbb F}_{K}(\cdot)\Bigl[\phi \bigl(\| H(\cdot+\tau ;x)\|_{Y}\bigr)\Bigr]_{|K\cap \Lambda} \Bigr\|_{P_{K}}\leq h_{K}\Bigl\|{\mathbb F}_{K}(\cdot)\Bigl[\phi\bigl(\| H(\cdot ;x)\|_{Y}\bigr)\Bigr]_{|K\cap \Lambda} \Bigr\|_{P_{K}}.
$$
\item[(f)] Any set $B$ of collection ${\mathcal B}$ is bounded.
\end{itemize}
If the function $F(\cdot;\cdot)$ is Levitan strongly $(\phi,{\mathbb F}_{\bf K},{\mathcal B},{\mathcal P}_{\bf K})$-almost periodic (Levitan semi-$(\phi,{\rm I},{\mathbb F}_{\bf K},{\mathcal B},{\mathcal P}_{\bf K})_{j\in {\mathbb N}_{n}}$-periodic), then  the function $F(\cdot;\cdot)$ is Levitan
$(\phi,{\mathrm R}, {\mathcal B},\phi,{\mathbb F}_{\bf K},{\mathcal P}_{\bf K})$-normal.
\item[(ii)] Suppose that $\emptyset  \neq \Lambda' \subseteq {\mathbb R}^{n},$ $\emptyset  \neq \Lambda \subseteq {\mathbb R}^{n},$ $F : \Lambda \times X \rightarrow Y$ is a given function and $\Lambda' \subseteq \Lambda''.$ If the function $F(\cdot;\cdot)$ is Levitan strongly $(\phi,{\mathbb F}_{\bf K},{\mathcal B},{\mathcal P}_{\bf K})$-almost periodic (Levitan semi-$(\phi,{\rm I},{\mathbb F}_{\bf K},{\mathcal B},{\mathcal P}_{\bf K})_{j\in {\mathbb N}_{n}}$-periodic), then  the function\\ $F(\cdot;\cdot)$ is Levitan $(\phi,{\mathbb F}_{\bf K},{\mathcal B},\Lambda',{\rm I},{\mathcal P}_{\bf K})$-almost periodic, provided that
the assumptions \emph{(a)-(f)} given in the formulation of \emph{(i)} hold.
\end{itemize}
\end{prop}

\begin{proof}
We will include all details of the proof of (i) for the sake of completeness, considering the class of  Levitan strongly $(\phi,{\mathbb F}_{\bf K},{\mathcal B},{\mathcal P}_{\bf K})$-almost periodic functions, only. Let $F(\cdot;\cdot)$ be such a function,
let $\epsilon>0,$ $B\in {\mathcal B},$ and let $({\bf b}_{k})_{k\in {\mathbb N}}$ belongs to ${\mathrm R}.$ Suppose, further, that $K$ is a non-empty compact set of $ {\mathbb R}^{n}$ such that $K\cap \Lambda \neq \emptyset.$ Then there exists a trigonometric polynomial $ P_{k}^{B,K}(\cdot;\cdot)$ such that
${\mathbb F}_{K}(\cdot)[ \phi(\|P_{k}^{B,K}(\cdot;x)-F(\cdot;x)\|_{Y})]_{| K\cap \Lambda} \in P_{K}$ for all $x\in X,$ and
\begin{align*}
\sup_{x\in B} \Biggl\| {\mathbb F}_{K}(\cdot) \Bigl[ \phi\Bigl(\bigl\|P_{k}^{B,K}(\cdot;x)-F(\cdot;x)\bigr\|_{Y}\Bigr)\Bigr]_{| K\cap \Lambda}\Biggr\|_{P_{K}}<\epsilon/3.
\end{align*}
Using (e), we get:   
\begin{align}\label{crisis}
\sup_{x\in B} \Biggl\| {\mathbb F}_{K}(\cdot) \Bigl[ \phi\Bigl(\bigl\|P_{k}^{B,K}(\cdot+\tau;x)-F(\cdot+\tau;x)\bigr\|_{Y}\Bigr)\Bigr]_{| K\cap \Lambda}\Biggr\|_{P_{K}}<h_{K}\epsilon/3
\end{align}
for all $\tau \in \Lambda''.$ The set $B$ is bounded due to (f), so that the Bochner criterion for almost periodic functions in ${\mathbb R}^{n}$ ensures that there exist a subsequence $({\bf b}_{k_{m}})_{m\in {\mathbb N}}$ of $({\bf b}_{k})_{k\in {\mathbb N}}$ and a natural number $m_{0}\in {\mathbb N}$ such that, for every positive integers $m',\ m''\geq m_{0},$ we have
$$
\sup_{x\in B}\Bigl\| P_{k}^{B,K}\bigl({\bf t}+{\bf b}_{k_{m'}};x\bigr)-P_{k}^{B,K}\bigl({\bf t}+{\bf b}_{k_{m''}};x\bigr)\Bigr\|_{Y}<\epsilon/3,\quad {\bf t}\in {\mathbb R}^{n}.
$$
Using (c)-(d), we obtain:
$$
\sup_{x\in B}\Biggl\| {\mathbb F}_{K}(\cdot)\Bigl[ \phi\Bigl( \Bigl\| P_{k}^{B,K}(\cdot+{\bf b}_{k_{m'}};x)-P_{k}^{B,K}(\cdot+{\bf b}_{k_{m''}};x)\Bigr\|_{Y}\Bigr) \Bigr]_{| K\cap \Lambda}\Biggr\|_{P_{K}}\leq  g_{K} \phi(\epsilon/3),
$$
for all positive integers $m',\ m''\geq m_{0}.$
Then the final conclusion follows from conditions (a)-(b), the estimate \eqref{crisis} and the following decomposition:
\begin{align*}
\Biggl\|{\mathbb F}_{K}(\cdot)&\Bigl[ \phi\Bigl( \bigl\| F(\cdot+{\bf b}_{k_{m'}};x)-F(\cdot+{\bf b}_{k_{m''}};x)\bigr\|_{Y} \Bigr)\Bigr]_{| K\cap \Lambda}\Biggr\|_{P_{K}} 
\\& \leq c'_{K}\Biggl[\Biggl\|{\mathbb F}_{K}(\cdot)\Bigl[ \phi\Bigl( \bigl\| F(\cdot+{\bf b}_{k_{m'}};x)-P(\cdot+{\bf b}_{k_{m'}};x)\bigr\|_{Y} \Bigr)\Bigr]_{| K\cap \Lambda}\Biggr\|_{P_{K}} 
\\&+\Biggl\|{\mathbb F}_{K}(\cdot)\Bigl[ \phi\Bigl( \bigl\| P(\cdot+{\bf b}_{k_{m'}};x)-P(\cdot+{\bf b}_{k_{m''}};x)\bigr\|_{Y} \Bigr)\Bigr]_{| K\cap \Lambda}\Biggr\|_{P_{K}} 
\\&+\Biggl\|{\mathbb F}_{K}(\cdot)\Bigl[ \phi\Bigl( \bigl\| P(\cdot+{\bf b}_{k_{m''}};x)-F(\cdot+{\bf b}_{k_{m''}};x)\bigr\|_{Y} \Bigr)\Bigr]_{| K\cap \Lambda}\Biggr\|_{P_{K}} \Biggr],
\end{align*}
holding with a certain positive real constant $c'_{K}>0.$
\end{proof}

The method proposed in the proof of \cite[Theorem 6.1.37]{nova-selected} does not work for Levitan $N$-almost periodic functions and uniformly Poisson stable functions and, because of that, we will omit here all details regarding the extensions of Levitan $N$-almost periodic functions and the extensions of  uniformly Poisson stable functions. Further on, the statement of \cite[Proposition 2.10]{chaouchi-metr} can be formulated in our new framework under certain very restrictive assumptions.
Before proceeding further, we will also note that the class of metrical semi-$(c_{j},{\mathcal B})_{j\in {\mathbb N}_{n}}$-periodic functions has been analyzed in \cite[Subsection 2.2]{chaouchi-metr} following the initial analysis of J. Andres and D. Pennequin in \cite{andres}. The Levitan and Bebutov classes of metrical semi-$(c_{j},{\mathcal B})_{j\in {\mathbb N}_{n}}$-periodic functions can be considered similarly; we leave all details concerning these topics to the interested readers.

\subsection{Levitan $(N,c)$-almost periodic functions and uniformly Poisson $c$-stable functions}\label{zajeg}

In this subsection, we will provide some examples and 
open questions 
concerning the Levitan $(N,c)$-almost periodic functions and the uniformly Poisson $c$-stable functions, where $c\in {\mathbb C} \setminus \{0\};$ for simplicity, we will always consider here the case in which $\Lambda={\mathbb R}^{n},$
$\phi(x) \equiv x,$ ${\mathbb F}_{\bf K}\equiv 1$ and $P_{K}=C(K)$ for each non-empty compact set $K\subseteq {\mathbb R}^{n}.$ Let $\rho=c{\rm I};$ then any 
Levitan $(\phi,{\mathbb F}_{\bf K},{\mathcal B},\Lambda',\rho,{\mathcal P}_{\bf K})$-almost periodic function
is simply called  Levitan $(N,\Lambda',c)$-almost periodic and any Bebutov $(\phi,{\mathbb F}_{\bf K},{\mathcal B},\Lambda',\rho,{\mathcal P}_{\bf K})$-uniformly recurrent function is simply called uniformly Poisson $(\Lambda',c)$-stable [we omit the term ``$\Lambda'$'' from the notation if $\Lambda'={\mathbb R}^{n}$].  

The proofs of \cite[Proposition 4.2.14, Proposition 7.1.13]{nova-selected} do not work for the Levitan/Bebutov concepts; because of that, we would like to ask the following:\vspace{0.1cm}

\noindent {\bf Problem.} Suppose that $c\in {\mathbb C}\setminus \{0,1\}.$ Can we find a bounded continuous function $F: {\mathbb R}^{n} \rightarrow {\mathbb C}$ which is Levitan $(N,c)$-almost periodic (uniformly $c$-Poisson stable) but not Levitan $(N,c^{2})$-almost periodic 
(uniformly $c^{2}$-Poisson stable)? \vspace{0.1cm}

The following important counterexample shows that there exists a 
non-trivial uniformly Poisson $c$-stable function $F : {\mathbb R}\rightarrow {\mathbb R}$ for any complex number $c\neq 0;$ this counterexample is based on the conclusions established in \cite[Lemma 3.5]{gejaka} and particularly shows that the requirements of \cite[Proposition 7.1.9]{nova-selected} does not imply $c=\pm 1 $ for uniformly Poisson $c$-stable functions:

\begin{example}\label{rasx}
Consider the function $\varphi : {\mathbb R}\rightarrow [0,\infty)$ given by
\begin{align*}
\varphi(t):=\sum_{k=1}^{\infty}\sin^{2}\Bigl( \frac{\pi t}{2^{k}}\Bigr),\quad t\in {\mathbb R}.
\end{align*}
In the above-mentioned lemma, 
E. Ait Dads, B. Es-sebbar and L. Lhachimi have proved that the function $\varphi(\cdot)$ is continuous and 
\begin{align}\label{dewsc}
\lim_{l\rightarrow +\infty}\varphi\bigl( t+2^{l}\bigr)=\lim_{l\rightarrow +\infty}\varphi\bigl( t-2^{l}\bigr)=\varphi(t)+\varphi(1),
\end{align}
pointwisely on ${\mathbb R}.$ We will first prove that the function $\varphi(\cdot)$ is Lipschitz continuous as well as that the convergence in \eqref{dewsc} is uniform on compact subsets of ${\mathbb R}.$ The uniform continuity of function $\varphi(\cdot)$ simply follows from the fact that the function $\sin^{2}\cdot$ is Lipschitz with the corresponding Lipschitz constant $L=1$ and the Lagrange mean value theorem, which shows that $|\varphi(x)-\varphi(y)|\leq \pi |x-y|,$ $x,\ y\in {\mathbb R}.$ Let $K=[a,b]\subseteq {\mathbb R}$ be a compact set. Then we have:
\begin{align*}
\varphi\bigl( t+2^{l}\bigr)=\sum_{k=1}^{l}\sin^{2}\Bigl( \frac{\pi t}{2^{k}}\Bigr)+\sum_{k=1}^{\infty}\sin^{2}\Bigl( \frac{\pi t}{2^{k+l}}+ \frac{\pi }{2^{k}}\Bigr),\quad t\in {\mathbb R}.
\end{align*}
Using this equality, the above-mentioned fact that the function $\sin^{2}\cdot$ is Lipschitz continuous with the corresponding Lipschitz constant $L=1$ and the Langrange mean value theorem, we simply get that the convergence in \eqref{dewsc} is uniform in $t\in K$ since:
$$
\sum_{k=l+1}^{\infty}\sin^{2}\Bigl( \frac{\pi t}{2^{k}}\Bigr)\leq \sum_{k=l+1}^{\infty} \frac{\pi^{2} (|a|+|b|)^{2}}{4^{k}},\quad k,\ l\in {\mathbb N}
$$
and
$$
\Biggl| \sum_{k=1}^{\infty}\sin^{2}\Bigl( \frac{\pi t}{2^{k+l}}+\frac{\pi}{2^{k}}\Bigr)- \sum_{k=1}^{\infty}\sin^{2}\Bigl( \frac{\pi }{2^{k}}\Bigr)\Biggr| \leq \pi (|a|+|b|)\sum_{k=l+1}2^{-k},\quad k,\ l\in {\mathbb N};
$$
here we have also used the elementary inequality $|\sin t| \leq |t|,$ $t\in {\mathbb R}.$ Therefore, we have that the function $\varphi(\cdot)$ is Bebutov $(x,1,\{2^{l} : l\in {\mathbb N}\},\rho,{\mathcal P}_{K})$-uniformly recurrent, where $P_{K}=C(K)$ for each non-empty compact subset $K$ of ${\mathbb R},$ $D(\rho):=[0,\infty)$ and $\rho(\varphi(t)):=\varphi(t)+\varphi(1)$ for all $t\in {\mathbb R}$ (note that the function $\varphi(\cdot)$ is surjective since it is continuous, $\varphi(0)=0$ and the function $\varphi(\cdot)$ is Besicovitch unbounded, which follows from the inequality $\varphi(t)\geq f(\pi t),$ $t\in {\mathbb R}$ with the function $f(\cdot)$ being defined through \eqref{jednac}).

Suppose now that $c\in {\mathbb C} \setminus \{0\}$ is fixed and consider any non-trivial continuous $(\varphi(1),c)$-periodic function $p : {\mathbb R}\rightarrow {\mathbb C}$, i.e., any non-trivial continuous function 
$p : {\mathbb R}\rightarrow {\mathbb C}$ such that $p(t+\varphi(1))=cp(t)$ for all $t\in {\mathbb R}.$ Then an elementary argumentation involving the equation \eqref{dewsc} and the uniform continuity of function $p(\cdot)$ on the interval $[-a,a],$ where $a=\max_{t\in K}(\varphi(t)+\varphi(1)+1),$
shows that the function $(p\circ \varphi)(\cdot)$ is uniformly Poisson $c$-stable, non-trivial and
$\lim_{l\rightarrow +\infty}(p\circ \varphi)(\cdot +2^{l})=(p\circ \varphi)(\cdot),$ uniformly on compacts of ${\mathbb R}.$

Let us finally consider the case $c=1$ and the function $p(t):=\sin (2\pi \varphi(t)/\varphi(1)),$ $t\in {\mathbb R}.$ Then $p(\cdot)$ is of period $\varphi(1),$ and we know that the function $(p\circ \varphi)(\cdot)$ is uniformly Poisson stable. At the end of this example,
we would like to ask whether the function $(p\circ \varphi)(\cdot)$ is uniformly recurrent (cf. also \cite[Proposition 3.6]{gejaka})?
\end{example}

Before proceeding to the next subsection, we would also like to ask whether we can construct a continuous Levitan $(N,c)$-almost periodic function
$f : {\mathbb R} \rightarrow {\mathbb R}$ for any complex number $c\neq 0?$

\subsection{Multi-dimensional Levitan $N$-almost periodic functions}\label{kura-tura}

Let us recall that a continuous function
$F : {\mathbb R}^{n} \rightarrow X$ is almost periodic if and only if for each finite real number $\epsilon>0$ there exist relatively dense sets $E^{j}$
in ${\mathbb R}$ ($1\leq j\leq n$) such that the set $E\equiv \prod_{j=1}^{n}E^{j}$ is contained in the set of all
$\epsilon$-almost periods of $F(\cdot);$ see, e.g., \cite[Introduction, pp. 31--32]{nova-selected}. Motivated by this fact, we 
will first introduce the following notion (let us only note here that we can similarly reconsider the notion of pre-Levitan almost periodicity):

\begin{defn}\label{srnic}
Suppose that $F : {\mathbb R}^{n} \rightarrow X$ is a continuous function. Then we say that the function $F(\cdot)$ is strongly Levitan $N$-almost periodic if and only if $F(\cdot)$ is pre-Levitan almost periodic and, for every real numbers
$N>0$ and $\epsilon>0,$ there exist a finite real number $\eta >0$ and the relatively dense sets $E_{\eta;N}^{j}$
in ${\mathbb R}$ ($1\leq j\leq n$) such that the set $E_{\eta;N}\equiv \prod_{j=1}^{n}E_{\eta;N}^{j}$ consists solely of
$(\eta,N)$-almost periods of $F(\cdot)$ and $E_{\eta;N} \pm E_{\eta;N} \subseteq E(\epsilon, N).$
\end{defn}

It is clear that any strongly Levitan $N$-almost periodic function $F : {\mathbb R}^{n} \rightarrow X$ is Levitan $N$-almost periodic as well as that the both notions coincide in the one-dimensional setting.

We continue this section by introducing the following notion, which has been proposed for the first time by V. A. Marchenko in \cite{marchenko} (1950, the one-dimensional setting):

\begin{defn}\label{marche}
Suppose that $F : {\mathbb R}^{n} \rightarrow X$ is a continuous function. Then we say that $F(\cdot)$
is
Levitan $N$-almost periodic of type $1$ if and only if for each $j\in {\mathbb N}_{n}$ there exists a real sequence $(\lambda_{k}^{j})$ such that, for every $N>0$ and $\epsilon>0,$ there exist numbers $k\in {\mathbb N}$ and $\delta>0$ such that any point $\tau=(\tau_{1},\tau_{2},...,\tau_{n})\in {\mathbb R}^{n}$ satisfying
$|\lambda_{l}^{j}\tau_{j}|\leq \delta \ \ (\mbox{mod} \ \ 2\pi)$ for all $l\in {\mathbb N}_{k}$ and $j\in {\mathbb N}_{n}$ is an $\epsilon,N$-almost period of the function $F(\cdot).$
\end{defn}

Now we will clarify the following result:

\begin{prop}\label{bengeder}
Consider the following statements:
\begin{itemize}
\item[(i)] $F(\cdot)$
is Levitan $N$-almost periodic of type $1.$
\item[(ii)] $F(\cdot)$ is strongly Levitan $N$-almost periodic.
\item[(iii)] $F(\cdot)$ is Levitan $N$-almost periodic.
\end{itemize}
Then \emph{(i)} $\Leftrightarrow$ \emph{(ii)} and \emph{(ii)} $\Rightarrow$ \emph{(iii)}.
\end{prop}

\begin{proof}
Using Kronecker's theorem, the Bogolyubov theorem and the argumentation given on \cite[pp. 54-55, 57]{188} we can simply prove, as in the one-dimensional setting, that the statements (i) and (ii) are equivalent. The statement (ii) $\Rightarrow$ (iii) is trivial.
\end{proof}

Further on, we can simply show that the Levitan $N$-almost periodic functions
of type $1$ form a vector space with the usual operations. Due to Proposition \ref{bengeder}, it follows that the strongly Levitan $N$-almost periodic form a vector space with the usual operations. 
Concerning this result, it is logical to ask whether the Levitan $N$-almost periodic functions
form a vector space with the usual operations as well as whether 
the all statements 
from the formulation of Theorem \ref{bengeder}
are mutually equivalent, i.e., whether the implication (iii) $\Rightarrow$ (i) is true? Concerning the last question, we would like to notice that
the arguments contained on p. 57 of \cite{188} can be repeated once more in order to see that this is true, provided that the following generalization of the Bogolyubov theorem holds true in the higher-dimensional setting:

\begin{itemize}
\item[(S)]
Suppose that $A$ is a relatively dense set in ${\mathbb R}^{n}$ and $\delta>0.$ Then there exist a real number $\eta>0$ and a finite set $\{\omega_{1},...,\omega_{k}\}$ in ${\mathbb R}^{n}$ such that
$$
\Bigl \{ \tau \in {\mathbb R}^{n} : \bigl| e^{i\omega_{l}\tau}-1\bigr|<\eta\mbox{ for all }l=1,....,k \Bigr\} \subseteq A-A+A-A+B_{\delta}.
$$
\end{itemize}

Unfortunately, the statement (S) does not hold in the case that $n\geq 2,$ as the following trivial counterexample shows:

\begin{example}\label{lalala}
Suppose that $A={\mathbb Z}^{n},$ $n\geq 2$ and $\delta \in (0,1/4);$ then $A$ is relatively dense in ${\mathbb R}^{n}.$ If (S) is true, then there exist a sufficiently small number $\eta_{0}>0$ and two real numbers $\omega_{1}$ and $\omega_{2}$ (take $k=1$ in (S)) such that $W=\{ \tau \in {\mathbb R}^{n} : |\omega_{1}\tau_{1}+...+\omega_{n}\tau_{n}| \leq \eta_{0} \mbox{ (mod }2\pi) \}\subseteq A-A+A-A+B_{\delta}=A+B_{\delta}.$ If $\omega_{1}=\omega_{2}=0,$ then $W={\mathbb R}^{n}$ and the contradiction is obvious; otherwise, $W$ is a hyperplane in ${\mathbb R}^{n}$ and this clearly implies that $W$ cannot be contained in $A+B_{\delta}=(A-A)+...+(A-A)+B_{\delta}$, where we have exactly $l$ addends in the above sum with the number $l\in {\mathbb N}\setminus \{1\}$ being arbitrarily chosen. 
\end{example}

Therefore, it is also important to ask whether the Bogolyubov theorem can be satisfactorily formulated in the higher-dimensional setting?

We close this subsection with the observation that we will not reconsider the approximation theorem for multi-dimensional Levitan $N$-almost periodic functions and the uniqueness theorem for multi-dimensional Levitan $N$-almost periodic functions here; see \cite[pp. 60--62]{levitan} for the corresponding results established in the one-dimensional setting.

\section{Applications to the abstract Volterra integro-differential equations}\label{bebprim}

In this section, we will present several new theoretical results and their applications to the abstract Volterra integro-differential equations and the classical partial differential equations.
We start with the following theme:

\subsection{The actions of infinite convolution products}\label{tackice}

In the one-dimensional setting, the analysis of existence and uniqueness of almost periodic type solutions  
for various classes of the abstract Volterra integro-differential equations without initial conditions leans heavily on the analysis of the qualitative features of 
the infinite convolution product
\begin{align}\label{pikford}
t 
\mapsto F(t):=\int^{t}_{-\infty} R(t-s)f(s)\, ds,\quad t\in {\mathbb R}.
\end{align}

The main result of this subsection reads as follows:

\begin{prop}\label{mkmkpik}
Let $\emptyset \neq \Lambda'\subseteq {\mathbb R},$ let $A$ be a closed linear operator on $X,$ and let $(R(t))_{t>  0}\subseteq L(X)$ be a strongly continuous operator family such that $R(t)A\subseteq AR(t)$ for all $t>0$  and
$\int_{0}^{\infty}\|R(t )\|\, dt<\infty .$ Let ${\mathcal P}_{K}=C(K)$ for each non-empty compact set $K$ of ${\mathbb R}.$
If $f : {\mathbb R} \rightarrow X$ is a bounded, continuous and Levitan $(x,\Lambda',A,{\mathcal P}_{K})$-almost periodic, resp. bounded, continuous and  Bebutov 
$(x,\Lambda',A,{\mathcal P}_{K})$-uniformly recurrent, and the function $Af : {\mathbb R}\rightarrow X$ is well defined and bounded, then the function $F: {\mathbb R} \rightarrow X,$ given by
\eqref{pikford},
is bounded, continuous and Levitan $(x,\Lambda',A,{\mathcal P}_{K})$-almost periodic, resp. bounded, continuous and Bebutov 
$(x,\Lambda',A,{\mathcal P}_{K})$-uniformly recurrent.
\end{prop}

\begin{proof}
We will consider only Bebutov 
$(x,\Lambda',A,{\mathcal P}_{K})$-uniformly recurrent functions.
It is clear that the function $F(\cdot)$ is well-defined and bounded since
$$
F(t):=\int_{0}^{\infty}R(s)f(t-s)\, ds,\quad t\in {\mathbb R},
$$
as well as $\int_{0}^{\infty}\|R(t )\|\, dt<\infty $  and the function $f(\cdot)$ is bounded. The continuity of $F(\cdot)$ simply follows from the dominated convergence theorem. Since $R(t)A\subseteq AR(t)$ for all $t>0$  and the function $Af(\cdot)$ is bounded, we have that $AF(t)=\int_{0}^{\infty}R(s)Af(t-s)\, ds$ for all $ t\in {\mathbb R}.$ Let $\epsilon>0,$ let $-\infty<a<b<+\infty,$ and let $K=[a,b].$ Further on, let $-\infty <c<a$ be such that 
\begin{align}\label{ok}
\bigl(\|f\|_{\infty}+\|Af\|_{\infty}\bigr)\cdot \int_{a-c}^{\infty}\| R(s)\|\, ds<\epsilon/2.
\end{align}
Then there exists a sequence $({\bf \tau}_{k})$ in $\Lambda'$ such that $\lim_{k\rightarrow +\infty} |{\bf \tau}_{k}|=+\infty$ and 
\begin{align}\label{sqz}
\Biggl(\int_{0}^{\infty}\|R(t )\|\, dt \Biggr) \cdot \sup_{t\in [c,b]}\bigl\| f(t+\tau_{k})-Af(t) \bigr\| <\epsilon/2.
\end{align}
The final conclusion simply follows from the estimates \eqref{ok}-\eqref{sqz} and the next computation, which holds for every $k\in {\mathbb N}$ and $t\in K:$ 
\begin{align*}
& \bigl\|  F(t+\tau_{k})-AF(t)\bigr\| =\Biggl\| \int^{t}_{-\infty} R(t-s)\bigl[ f(s+\tau_{k})-Af(s) \bigr]\, ds  \Biggr\|
\\ & \leq \int^{c}_{-\infty} \| R(t-s)\| \cdot \bigl\| f(s+\tau_{k})-Af(s) \bigr\| \, ds+ \int^{t}_{c} \| R(t-s)\| \cdot \bigl\| f(s+\tau_{k})-Af(s) \bigr\| \, ds
\\& \leq \bigl(\|f\|_{\infty}+\|Af\|_{\infty}\bigr)\cdot  \int_{a-c}^{\infty}\| R(s)\|\, ds 
\\&+ \Biggl(\int^{c}_{t}\| R(t-s)\|\, ds \Biggr)  \cdot \sup_{s\in [c,b]}\bigl\| f(s+\tau_{k})-Af(s) \bigr\| \leq (\epsilon/2)+(\epsilon/2)=\epsilon.
\end{align*}
\end{proof}

\begin{rem}\label{ejha}
\begin{itemize}
\item[(i)] Suppose that $\nu_{K} : K \rightarrow (0,\infty)$ satisfies that the function $1/\nu_{K}(\cdot)$ is bounded. Then the 
use of weighted function space
$C_{b,\nu_{K}}(K )$ seems to be discutable here since the argumentation contained in the proof given above indicates that the function $\nu_{K}(\cdot)$ has to be bounded, as well, unless we assume certain very restrictive conditions.
\item[(ii)] The statement of Proposition \ref{mkmkpik} can be transferred to the multi-dimensional setting without any serious difficulty; cf. \cite{nova-selected} for more details.
\item[(iii)] For simplicity, let us consider here the class of Levitan $(x,{\mathbb R},{\rm I},{\mathcal P}_{K})$-almost periodic functions. If we additionally suppose that for each $N>0$ and $\epsilon>0$ there exists a 
relatively dense set $E_{\eta;N}$ of $(\eta,N)$-almost periods of $f(\cdot)$ such that $E_{\eta;N} \pm E_{\eta;N} \subseteq E(\epsilon, N),$ then the same property holds for the function $F(\cdot).$
\end{itemize}
\end{rem}

As mentioned many times before, our results about the convolution invariance of generalized metrical Levitan almost periodicity and 
generalized metrical Bebutov uniform recurrence under the actions of the infinite convolution product \eqref{pikford} 
can be successfully applied in the qualitative analysis of solutions for a large class of the abstract (degenerate) Volterra integro-differential equations without initial conditions. For example, Proposition \ref{mkmkpik} is clearly applicable in the analysis of the existence and uniqueness of Bebutov uniformly recurrent solutions of 
the initial value problems with constant coefficients
\[
\begin{array}{l}
D_{t,+}^{\gamma}u(t,x)=\sum_{|\alpha|\leq k}a_{\alpha}f^{(\alpha)}(t,x)+f(t,x),\ t\in {\mathbb R},\ x\in \mathbb{R}^n
\end{array}
\] 
in the space $L^{p}(\mathbb{R}^n),$ where 
$\gamma \in (0,1),$
$D_{t,+}^{\gamma}u(t)$ denotes the Weyl-Liouville fractional derivative of order $\gamma$ and
$1\leq p<\infty ;$
cf. also \cite{bloch} and \cite{nova-mono} for many similar applications of this type. 

\subsection{The convolution invariance and some applications}\label{216}

In this subsection, we continue our investigation from the previous subsection by examining the convolution invariance of Levitan almost like periodicity. In actual fact, we would like to observe that
the method proposed in the proof of Proposition \ref{mkmkpik} enables one to simply deduce the following result:

\begin{prop}\label{mkmkpik1}
Let $\emptyset \neq \Lambda'\subseteq {\mathbb R}^{n},$ let $A$ be a closed linear operator on $X,$ $h\in L^{1}({\mathbb R}^{n}),$ and let ${\mathcal P}_{K}=C(K)$ for each non-empty compact set $K$ of ${\mathbb R}.$
If $f : {\mathbb R}^{n} \rightarrow X$ is bounded, continuous and Levitan $(x,\Lambda',A,{\mathcal P}_{K})$-almost periodic, resp. bounded, continuous and Bebutov 
$(x,\Lambda',A,{\mathcal P}_{K})$-uniformly recurrent, and the function $Af : {\mathbb R}^{n}\rightarrow X$ is well defined and bounded, then the function $H: {\mathbb R}^{n} \rightarrow X,$ given by
$$
H({\bf t}):= (h\ast f)({\bf t}):= \int_{{\mathbb R}^{n}}h({\bf t}-{\bf s})f({\bf s})\, d{\bf s},\quad {\bf t}\in {\mathbb R}^{n},
$$
is bounded, continuous and Levitan $(x,\Lambda',A,{\mathcal P}_{K})$-almost periodic, resp. bounded, continuous and Bebutov 
$(x,\Lambda',A,{\mathcal P}_{K})$-uniformly recurrent.
\end{prop}

\begin{rem}\label{ejha1}
\begin{itemize}
\item[(i)] Unfortunately, we cannot formulate a satisfactory analogue of Proposition \ref{mkmkpik1} in the case that there exist two finite real numbers $b>0$ and $c>0$ such that $\| f({\bf t})\|\leq c(1+|{\bf t}|)^{b},$ ${\bf t}\in {\mathbb R}^{n}$ and $h(\cdot)[1+|\cdot|]^{b}\in L^{1}({\mathbb R}^{n}),$ when the function $H(\cdot)$ is clearly well-defined. 
\item[(ii)] Let us consider the class of Levitan $(x,{\mathbb R},{\rm I},{\mathcal P}_{K})$-almost periodic functions. If we suppose, additionally, that for each $N>0$ and $\epsilon>0$ there exists a 
relatively dense set $E_{\eta;N}$ of $(\eta,N)$-almost periods of $f(\cdot)$ such that $E_{\eta;N} \pm E_{\eta;N} \subseteq E(\epsilon, N),$ then the same property holds for the function $H(\cdot).$ Taken together with the statement of Proposition \ref{mkmkpik1}, this provides a proper extension of the well known result \cite[Theorem 13]{sjerp}, proved for the first time by D. Bugajewski, X. Gan and P. Kasprzak. 
\end{itemize}
\end{rem}

Now we will briefly explain how we can incorporate Proposition \ref{mkmkpik1} in the analysis of the existence and uniqueness of the
Levitan almost periodic type solutions for some classes of (fractional) partial differential equations:\vspace{0.1cm}

1. We will consider first
the inhomogeneous heat equation in ${\mathbb R}^{n}$.
Recall,
a unique solution of the heat equation $u_{t}(t,x)=u_{xx}(t,x),$ $t\geq 0,\ x\in {\mathbb R}^{n};$ $u(0,x)=F(x),$ $x\in {\mathbb R}^{n}$ is given by the action of 
Gaussian semigroup 
\begin{align*}
F\mapsto (G(t)F)(x)\equiv \bigl(4\pi t\bigr)^{-n/2}\int_{{\mathbb R}^{n}}e^{-|y|^{2}/4t}F(x-y)\, dy,\quad t>0,\ x\in {\mathbb R}^{n}.
\end{align*}
Suppose now that $c\in {\mathbb C} \setminus \{0\},$ a number $t>0$ is fixed and the function $F(\cdot)$ is bounded, continuous and Levitan $(x,\Lambda',c,{\mathcal P}_{K})$-almost periodic, resp. bounded, continuous and Bebutov 
$(x,\Lambda',c,{\mathcal P}_{K})$-uniformly recurrent. Then an application of Proposition \ref{mkmkpik1} shows that the function $x\mapsto (G(t)F)(x),$ $x\in {\mathbb R}^{n}$ is likewise bounded, continuous and Levitan $(x,\Lambda',c,{\mathcal P}_{K})$-almost periodic, resp. bounded, continuous and Bebutov 
$(x,\Lambda',c,{\mathcal P}_{K})$-uniformly recurrent (cf. also S. Zaidman \cite[Examples 4, 5, 7, 8; pp. 32-34]{30}, which can be also used for our purposes).

2. Suppose that $1<\beta<2.$ The fractional diffusion-wave equation\index{equation!fractional diffusion-wave}
$$
{\mathbb D}_{t}^{(\beta)}u(t,x)=\Delta u(t,x),\quad t>0,\ x\in {\mathbb R}^{n},
$$
subjected with the initial conditions
$$
u(0, x) = u_{0}(x),\ \ u_{t}(0, x) = 0,
$$
where ${\mathbb D}_{t}^{(\beta)}u(t,x)$
is the Caputo-Dzhrbashyan fractional derivative, given by\index{fractional derivative!Caputo-Dzhrbashyan} {\small
$$
{\mathbb D}_{t}^{(\beta)}u(t,x):=\frac{1}{\Gamma(2-\beta)}\frac{\partial^{2}}{\partial^{2}t}\int^{t}_{0}(t-\tau)^{1-\beta}u(\tau,x)\, d\tau-t^{1-\beta}\frac{u_{t}(0,x)}{\Gamma(2-\beta)}-t^{1-\beta}\frac{u(0,x)}{\Gamma(1-\beta)},\quad t>0,
$$}
was studied by many authors (see, e.g., the research article \cite{kochubei} by A. N. Kochubei and references cited therein). A
unique solution $u(t,x)$ of this problem is obtained as the convolution with a Green kernel
$$
u(t,x)=\int_{{\mathbb R}^{n}}G(t,x-\xi)u_{0}(\xi)\, d\xi=\int_{{\mathbb R}^{n}}G(t,\xi)u_{0}(x-\xi)\, d\xi.
$$
Since the Green kernel $G(t,x)$ satisfies the estimate (cf. the research article \cite{izvestija} by A. V. Pskhu for more details; $c>0$ and $a>0$ denote the positive real constants here):
$$
|G(t,x)|\leq ct^{-n\beta/2}\gamma_{n}\bigl( |x| t^{-\beta/2}\bigr)E\bigl( |x| t^{-\beta/2}\bigr),
$$
where $\gamma_{n}(z):=1,$ if $n=1$ [$\gamma_{n}(z):=|\ln z|,$ if $n=2;$ $\gamma_{n}(z):=z^{2-n}$, if $n\geq 3$] and $E(z):=\exp(-az^{2/(2-\beta)}),$ it can be easily shown the following: If the function $x\mapsto u_{0}(x),$ $x\in {\mathbb R}^{n}$ is bounded, continuous and $(I',c)$-almost periodic ($(I',c)$-uniformly recurrent; Levitan $(x,\Lambda',c,{\mathcal P}_{K})$-almost periodic, Bebutov 
$(x,\Lambda',c,{\mathcal P}_{K})$-uniformly recurrent) for some $c\in {\mathbb C},$ then the solution $u(t,x)$ is bounded, continuous and $(I',c)$-almost periodic ($(I',c)$-uniformly recurrent; Levitan $(x,\Lambda',c,{\mathcal P}_{K})$-almost periodic, Bebutov 
$(x,\Lambda',c,{\mathcal P}_{K})$-uniformly recurrent) in the variable $x$ for every fixed value of variable $t>0;$ cf. \cite{rho} for the notion used.

3. Using the Laplace transform of vector-valued distributions, M. Kunzinger, E. A. Nigscha and N. Ortner \cite{kunzinger} have presented the explicit solutions for the Cauchy-Dirichlet problem of the iterated wave operator
$
( \Delta_{n}+\partial_{y}^{2}-\partial_{t}^{2} )^{m},
$
the iterated Klein-Gordon operator
$
( \Delta_{2n+1}+\partial_{y}^{2}-\partial_{t}^{2}-\xi^{2})^{m},
$
the iterated metaharmonic operator
$
( \Delta_{n}+\partial_{y}^{2}-p^{2} )^{m},
$
and the iterated heat operator
$
( \Delta_{n}+\partial_{y}^{2}-\partial_{t})^{m},
$
in the half-space $y>0.$ In particular, the authors have shown that, under certain logical assumptions, the Cauchy problem
\begin{align*}
\Bigl( & \partial_{x_{1}}^{2}+...+\partial_{x_{n}}^{2}+\partial_{y}^{2} \Bigr)^{2}=0\ \ \mbox{ in }\ \ (x,y)\in {\mathbb R}^{n} \times (0,\infty);
\\& u_{| y=0}=g_{0}(x),\ \bigl(\partial u_{y}\bigr)_{| y=0}=g_{1}(x)
\end{align*}
has a unique solution
\begin{align*}
u(x,y)=2\frac{\Gamma((n+3)/2)}{\pi^{(n+1)/2}}y^{3}&\int_{{\mathbb R}^{n} }g_{0}(x-\xi)\frac{d\xi}{\bigl( |\xi|^{2}+y^{2} \bigr)^{(n+3)/2}} \\&+\frac{\Gamma((n+1)/2)}{\pi^{(n+1)/2}}y^{2}\int_{{\mathbb R}^{n} }g_{1}(x-\xi)\frac{d\xi}{\bigl( |\xi|^{2}+y^{2} \bigr)^{(n+1)/2}} .
\end{align*}
Suppose now that the function $x\mapsto (g_{0}(x),g_{1}(x)),$ $x\in {\mathbb R}^{n}$ is bounded, continuous and $(I',c)$-almost periodic ($(I',c)$-uniformly recurrent; Levitan $(x,\Lambda',c,{\mathcal P}_{K})$-almost periodic, Bebutov 
$(x,\Lambda',c,{\mathcal P}_{K})$-uniformly recurrent) for some $c\in {\mathbb C}.$ Then it can be easily shown that the solution $u(x,y)$ is bounded, continuous and $(I',c)$-almost periodic ($(I',c)$-uniformly recurrent; Levitan $(x,\Lambda',c,{\mathcal P}_{K})$-almost periodic, Bebutov 
$(x,\Lambda',c,{\mathcal P}_{K})$-uniformly recurrent) in the variable $x$ for every fixed value of variable $y>0.$ This example and the first example given in this subsection can be also used to justify the introduction of multi-dimensional Levitan $\rho$-almost periodic type functions with $\rho$ being not a function but a general binary relation; cf. \cite{rho} for more details concerning this issue.

4. We can similarly consider the existence and uniqueness of Levitan almost periodic type solutions of the inhomogeneous abstract Cauchy problems considered in \cite[Example 1.1]{marko-manuel-ap}, the backward wave equation considered in \cite[Example 8, p. 33]{30},
and provide certain applications to the abstract ill-posed Cauchy problems considered on \cite[pp. 543-545]{nova-selected}.

5. Concerning semilinear Cauchy problems, we will only note that the composition principle clarified in \cite[Theorem 7.1.18]{nova-selected} can be simply formulated for the uniformly Poisson $c$-stable functions. Keeping in  mind Proposition \ref{mkmkpik1}, we can simply consider the existence and uniqueness of the uniformly Poisson $c$-stable solutions for the class of semilinear Hammerstein
integral equation of convolution type on ${\mathbb R}^{n}$; see \cite[p. 470]{nova-selected} and the corresponding application made in \cite[Section 3]{sds}. Without going into full details, we will only note that the statement of \cite[Theorem 12]{sjerp} and some composition principles clarified in \cite[Subsection 6.1.5]{nova-selected} can probably be formulated in our new framework.

6. Before proceeding to the next subsection, we would like to notice that the convolution invariance of Levitan $N$-almost periodicity has recently been considered by A. Nawrocki in \cite[Theorem 4.2, Theorem 4.4, Theorem 4.8]{nawr}. Concerning these statements, clarified in the one-dimensional setting, we would like to make the following comments:

\begin{itemize}
\item[(i)] Theorem 4.2 can be formulated for the uniformly Poisson stable (Levitan $N$-almost periodic) functions $f : {\mathbb R} \rightarrow Y.$ Then the convolution $f\ast g$ exists and it is a  uniformly Poisson stable (Levitan $N$-almost periodic) function.
\item[(ii)] Theorem 4.4 can be formulated for the Stepanov bounded, uniformly Poisson stable (Levitan $N$-almost periodic) functions $f : {\mathbb R} \rightarrow Y.$ The resulting convolution $f\ast g_{\lambda}$ exists and it is a bounded, uniformly Poisson stable (Levitan $N$-almost periodic) function.
\item[(iii)] Concerning Theorem 4.8, we will only note that the Stepanov unboundedness condition \cite[(4.2)]{nawr} is not suitable for the work with the uniformly Poisson stable functions. Let us assume, in place of this condition, that for each real number $\omega>0$ and for each strictly increasing sequence $(\alpha_{k})$ of positive integers we have 
$$
\limsup_{k\rightarrow +\infty}\int^{(a_{k}+1)\omega}_{a_{k}\omega}f(t)\, dt=+\infty.
$$ 
Then the resulting convolution $f\ast g_{\lambda}$ exists and it is a bounded, uniformly Poisson stable function. The statement can be also formulated for the vector-valued functions.
\end{itemize}

In \cite[Section 5]{nawr}, the author has applied the above results (cf. also \cite[Corollary 4.11-Corollary 4.14]{nawr})
in the analysis of the existence and uniqueness of Levitan $N$-almost periodic solutions of the linear ordinary differential equation of first order
$$
y^{\prime}(x)=\lambda y(x)+f(x),\quad x\in {\mathbb R}.
$$
The
statements of \cite[Lemma 5.5, Theorem 5.10(i)]{nawr} can be formulated for the uniformly Poisson stable solutions of this equation.

\subsection{On the wave equation in ${\mathbb R}^{n}$}\label{21666}

In this subsection, we will revisit our recent investigations of the wave equation in ${\mathbb R}^{n}$, where $n\leq 3.$
Let us consider first the following wave equation in ${\mathbb R}^{3}:$
\begin{align}\label{relax-wave}
u_{tt}(t,x)=d^{2}\Delta_{x}u(t,x),\quad x\in {\mathbb R}^{3},\ t>0;
\ \ u(0,x)=g(x),\ \ u_{t}(0,x)=h(x),
\end{align}
where $d>0,$ $g\in C^{3}({\mathbb R}^{3}: {\mathbb R})$ and  $h\in C^{2}({\mathbb R}^{3}: {\mathbb R}).$ By the Kirchhoff formula (see e.g., \cite[Theorem 5.4, pp. 277-278]{salsa}), the function\index{Kirchhoff formula}
\begin{align*}
 u(t,x)& :=\frac{\partial}{\partial t}\Biggl[ \frac{1}{4\pi d^{2}t}\int_{\partial B_{dt}(x)}g({\bf \sigma})\, d\sigma \Biggr]+\frac{1}{4\pi d^{2}t}\int_{\partial B_{dt}(x)}g({\bf \sigma})
\, d\sigma
\\ & =\frac{1}{4\pi}\int_{\partial B_{1}(0)}g(x+dt{\bf \omega})\, d\omega
+\frac{dt}{4\pi}\int_{\partial B_{1}(0)}\nabla g(x+dt{\bf \omega}) \cdot {\bf \omega}\, d\omega
\\& +\frac{t}{4\pi}\int_{\partial B_{1}(0)}h(x+dt{\bf \omega})\, d\omega
,\quad t\geq 0,\ x\in {\mathbb R}^{3},
\end{align*}
is a unique solution of problem \eqref{relax-wave} which belongs to the class $C^{2}([0,\infty)\times {\mathbb R}^{3})$. Fix now a number $t_{0}>0.$ Then the function $x\mapsto u(t_{0},x),$ $x\in {\mathbb R}^{3}$ will be Levitan $N$-almost periodic (uniformly Poisson stable) provided that the function $(g(\cdot),\ \nabla g(\cdot), h(\cdot))$
is  Levitan $N$-almost periodic (uniformly Poisson stable), for example.

Let us consider now the following wave equation in ${\mathbb R}^{2}:$\index{Poisson formula}
\begin{align}\label{relax-wave2}
u_{tt}(t,x)=d^{2}\Delta_{x}u(t,x),\quad x\in {\mathbb R}^{2},\ t>0;
\ \ u(0,x)=g(x),\ \ u_{t}(0,x)=h(x),
\end{align}
where $d>0,$ $g\in C^{3}({\mathbb R}^{2}: {\mathbb R})$ and  $h\in C^{2}({\mathbb R}^{2}: {\mathbb R}).$ By the Poisson formula (see e.g., \cite[Theorem 5.5, pp. 280-281]{salsa}), we know that the function{\small
\begin{align*}
u(t,x)& :=\frac{\partial}{\partial t}\Biggl[ \frac{1}{2\pi d}\int_{\partial B_{dt}(x)}\frac{g({\bf \sigma})}{\sqrt{d^2t^{2}-|x-y|^{2}}}\, d\sigma \Biggr]+\frac{1}{2\pi d}\int_{\partial B_{dt}(x)}\frac{h({\bf \sigma})}{\sqrt{d^2t^{2}-|x-y|^{2}}}\, d\sigma
\\& =d\int_{B_{1}(0)}\frac{g(x+dt\sigma)}{\sqrt{1-|{\sigma}|^{2}}}\, d\sigma +d^{2}t\int_{B_{1}(0)}\frac{\nabla g(x+dt\sigma) \cdot \sigma}{\sqrt{1-|{\sigma}|^{2}}}\, d\sigma
\\& +  dt\int_{B_{1}(0)}\frac{h(x+dt\sigma)}{\sqrt{1-|{\sigma}|^{2}}}\, d\sigma,\quad t\geq 0,\ x\in {\mathbb R}^{2},
\end{align*}}
is a unique solution of problem \eqref{relax-wave2} which belongs to the class $C^{2}([0,\infty) \times {\mathbb R}^{3})$. Fix now a number $t_{0}>0.$ Similarly as above, we have that the function $x\mapsto u(t_{0},x),$ $x\in {\mathbb R}^{2}$ will be Levitan $N$-almost periodic (uniformly Poisson stable) provided that the function $(g(\cdot),\ \nabla g(\cdot), h(\cdot))$
is  Levitan $N$-almost periodic (uniformly Poisson stable).

Let us finally consider the wave equation  $u_{tt}=a^{2}u_{xx}$ in domain $\{(x,t) : x\in {\mathbb R},\ t>0\},$ equipped with the initial conditions $u(x,0)=f(x)\in C^{2}({\mathbb R})$ and $u_{t}(x,0)=g(x)\in C^{1}({\mathbb R})$. As is well known, its unique regular 
solution
is given by
the famous d'Alembert formula
$$
u(x,t)=\frac{1}{2}\bigl[ f(x-at) +f(x+at) \bigr]+\frac{1}{2a}\int^{x+at}_{x-at}g(s)\, ds,\quad x\in {\mathbb R}, \ t>0.
$$  
Our consideration from \cite[Example 1.2]{marko-manuel-ap} shows that the solution 
$u(x,t)$ can be extended to the whole real line in time variable and that $u(x,t)$ will be uniformly $c$-Poisson stable in $(x,t)\in {\mathbb R}^{2},$ provided that the function
$t\mapsto (f(\cdot),\int^{\cdot}_{0}g(s)\, ds),$ $t\in {\mathbb R}$ is uniformly $c$-Poisson stable, with the meaning clear ($c\in {\mathbb C}\setminus \{0\}$).
We would like to emphasize that it is not clear how one can prove that the solution $u(x,t)$ will be Levitan $N$-almost periodic in $(x,t)\in {\mathbb R}^{2},$ provided that the functions
 $f(\cdot)$ and $\int^{\cdot}_{0}g(s)\, ds$ are Levitan $N$-almost periodic.

\section{Conclusions and final remarks}\label{toi}

In this paper, we have investigated the Levitan and Bebutov approaches to the metrical approximations by trigonometric polynomials and $\rho$-periodic type functions, providing also certain applications to
the abstract Volterra integro-differential equations and the partial differential equations. We have proved many structural results for the introduced classes of functions and propose several open problems, remarks and illustrative examples.

Finally, we would like to mention some topics not considered in our former work:\vspace{0.1cm}

1. The Levitan and Bebutov approaches to the metrical approximations by trigonometric polynomials and $\rho$-periodic type functions can be further generalized using the approaches of Stepanov, Weyl and Besicovitch. Many structural results established in \cite{chaouchi-metr} can be reformulated in this context.\vspace{0.1cm}

2. It is well known that every Levitan $N$-almost periodic function $f: {\mathbb R} \rightarrow {\mathbb C}$ can be represented by the
uniform limit of ratios of scalar-valued almost periodic functions (see A. G. Baskakov \cite{baskakov}). We will not discuss here the question whether this result continues to hold for multi-dimensional 
Levitan $N$-almost periodic functions.\vspace{0.1cm}

3. In this paper, we have not analyzed ${\mathbb D}$-asymptotically Levitan (Bebutov) almost periodic type functions in general metric. Concerning this issue, we will only note that we have recently 
provided some new results about the existence and uniqueness of ${\mathbb D}$-asymptotically almost automorphic solutions of 
the inhomogeneous wave equation 
\begin{align*}
\notag & u_{tt}(t,x)-d^{2}\Delta_{x}u(t,x)=f(t,x),\quad x\in {\mathbb R}^{2},\ t>0;
\\  & u(0,x)=g(x),\ \ u_{t}(0,x)=h(x),
\end{align*}
where $d>0,$ $f(t,x)$ is continuously differentiable in the variable $t\in {\mathbb R}$ and continuous in the variable $x\in {\mathbb R},$
$g\in C^{2}({\mathbb R}^{2}: {\mathbb R})$ and $h\in C^{1}({\mathbb R}^{2}: {\mathbb R}).$ It is well known that the unique solution of this problem is given by
\begin{align*}
u(x,t)&=\frac{1}{2}\bigl[ g(x-at) +g(x+at) \bigr]+\frac{1}{2a}\int^{x+at}_{x-at}h(s)\, ds
\\& +\frac{1}{2d}\int^{t}_{0}\biggl[ \int^{x+d(t-s)}_{x-d(t-s)}f(r,s) \, dr \Biggr]\, ds,\quad x\in {\mathbb R}, \ t>0,
\end{align*}
so that the conclusions established on \cite[pp. 541-542]{nova-selected} can be also formulated for ${\mathbb D}$-asymptotically uniformly Poisson stable solutions (defined in the obvious way), for example. Details can be left to the interested readers.

\end{document}